\documentclass[11pt, a4paper]{article}

\usepackage{amsmath}
\usepackage{amsthm}
\usepackage{amsfonts}
\usepackage{amssymb}
\usepackage{amscd}
\usepackage[dvips]{graphicx, color}

\newtheorem{thm}{Theorem}[section]
\newtheorem{cor}[thm]{Corollary}

\newtheorem{lem}[thm]{Lemma}
\newtheorem{prop}[thm]{Proposition}
\newtheorem{rem}[thm]{Remark}
\newtheorem{exam}[thm]{Example}

\numberwithin{equation}{section}

\title{Braided Morita equivalence for finite-dimensional semisimple and cosemisimple Hopf algebras
\centerline{\small \it Dedicated to Professor Masahiko Suzuki on the occasion of his 65th birthday}}
\author{Michihisa \textsc{Wakui}\\ 
{\footnotesize  Department of Mathematics, Faculty of Engineering Science,}\\ 
{\footnotesize  Kansai University, Suita-shi, Osaka 564-8680, Japan}\\
{\footnotesize  e-mail address: wakui@kansai-u.ac.jp}\\ 
}
\date{}

\begin{document}

\maketitle 
\begin{abstract}
Braided Morita invariants of finite-dimensional semisimple and cosemisimple Hopf algebras with braidings are constructed by refining the polynomial invariants introduced by the author. 
The invariants are computed for the duals of Suzuki's braided Hopf algebras, 
and as an application of that, the braided Morita equivalence classes over the $8$-dimensional Kac-Paljutkin algebra are determined. 
This paper also includes the modified results and proofs on determination of the coribbon elements of Suzuki's braided Hopf algebras, that are discussed and given in \cite{W2}. 
\end{abstract}

\baselineskip 15pt 

\section{Introduction}
On the classification of Hopf algebras over a field $\boldsymbol{k}$ two problems are now actively progressed. 
One is the classification up to isomorphism under some restriction like dimension fixed, or semisimple, or pointed. 
Another is the classification up to monoidal Morita equivalence, that is based on a categorical point of view. 
Two Hopf algebras $A$ and $B$ are called a $\boldsymbol{k}$-linear monoidal Morita equivalent if their module categories ${}_A\mathbb{M}$ and ${}_B\mathbb{M}$ are equivalent as $\boldsymbol{k}$-linear monoidal categories. 
In this paper we take the later stance, and consider some classification problem on quasitriangular Hopf algebras, namely, Hopf algebras with braiding structures. 
A braiding structure on a Hopf algebra $A$ is determined by some element $R\in A\otimes A$ called a universal $R$-matrix, which is introduced by Drinfeld \cite{Dri}. 
We write $c^R$ for the braiding structure and ${}_{(A,R)}\mathbb{M}$ for the braided monoidal category  $({}_A\mathbb{M}, c^R)$. 
Two quasitriangular Hopf algebras $(A, R)$ and $(B, R^{\prime})$ are called braided Morita equivalent 
if the braided categories ${}_{(A,R)}\mathbb{M}$ and ${}_{(B,R^{\prime})}\mathbb{M}$ are equivalent as 
$\boldsymbol{k}$-linear braided monoidal categories. 
There are a few results of classification of quasitriangular Hopf algebras up to braided Morita equivalence \cite{GMN, NN}. 
\par 
The eigenvalues of $S$-matrices and the Brauer groups in a braided monoidal category are well-known as braided Morita invariants \cite{KrillovJr, VanOZ}. 
In \cite{W} the author introduced some monoidal Morita invariant of semisimple and cosemisimple Hopf algebras of finite dimension. 
It is given as a polynomial in one variable, which constructed from the data of the braidings and the absolutely simple modules. 
By refining the invariant on braidings we have braided Morita invariants of semisimple and cosemisimple quasitriangular Hopf algebras of finite dimension. 
In this paper we compute these braided Morita invariants for the duals of Suzuki's braided Hopf algebras \cite{Suzuki2}, which fit into a Hopf algebra extension 
$1 \longrightarrow (\boldsymbol{k}C_2)^{\ast} \longrightarrow K \longrightarrow \boldsymbol{k}D_{2L} \longrightarrow 1$, where $C_2$ is the cyclic group of order $2$, and $D_{2L}$ is the dihedral group of order $2L$. 
In particular, the $8$-dimensional Kac-Paljutkin algebra \cite{KP, Masuoka0}, denoted by $H_8$, is contained in the family of Suzuki's Hopf algebras. 
As an application of the computation results of our polynomial invariants, we determine the braided Morita equivalence classes over $H_8$. 
\par 
In closely connection with the above consideration, the coribbon elements of Suzuki's braided Hopf algebras are determined.
Actually, although they have studied in \cite{W2} by the author, 
the proof of Lemma 8 and the statement of Theorem 5  in \cite{W2} contain several mistakes. 
I noticed them by a detailed note \cite{Sommer} sent from Sommerh\"{a}user. 
We modify arguments  in \cite{W2} and show the correct results on that with thanks to him. 
Another proof of the revised version of Theorem 5  in \cite{W2} 
is also given by using the spherical structures of Suzuki's Hopf algebras. 
\par 
This paper is organized as follows. 
In Section 2 we review the definition of (co)ribbon Hopf algebras, 
and introduce braided Morita invariants of semisimple and cosemisimple quasitriangular Hopf algebras of finite dimension. 
In Section 3  we review the definition of Suzuki's braided Hopf algebras and some basic results on that obtained by Satoshi Suzuki \cite{Suzuki2}. 
We give the revised results on determination of the coribbon elements of Suzuki's braided Hopf algebras. 
In Section 4 we compute the polynomial invariants defined in Section 2 for the duals of Suzuki's braided Hopf algebras. 
In the final section we compute the Hopf algebra automorphism group for $H_8$, 
and determine the braided Morita equivalence classes of $H_8$. 
In Appendix we give a list of corrigenda in my paper \cite{W2}. 
\par 
Throughout this paper $\boldsymbol{k}$ denotes a field. 
For a bialgebra or a Hopf algebra $A$, denoted by $\Delta $, $\varepsilon $ and $S$ the comultiplication, the counit and the antipode of $A$, respectively. 
We use Sweedler's notation such as $\Delta (x)=\sum x_{(1)}\otimes x_{(2)}$ for $x\in A$. 
For general facts on Hopf algebras or monoidal categories, refer to Montgomery's book \cite{MontgomeryBook} and Kassel's book \cite{Kassel}. 

\par \noindent 
{\bf Acknowledgments}.  
I express my sincere gratitude to Professor Yorck Sommerh\"{a}user for careful reading  my paper \cite{W2} and for  reminding me that there are incorrect descriptions in it.  
I would like to thank Professor Hiroyuki Yamane for giving an opportunity to speak in this conference and write a paper in the proceedings. 
I would also like to thank the referee for helpful comments on improving this paper.

\section{Braided Morita invariants of quasitriangular Hopf algebras}
\subsection{Definitions of braided and coribbon Hopf algebras} 
The notion of a quasitriangular bialgebra or a quasitriangular Hopf algebra is introduced by Drinfeld \cite{Dri}. 
It is a pair of a bialgebra or a Hopf algebra $A$ over $\boldsymbol{k}$ and an invertible element $R\in A\otimes A$ satisfying some suitable conditions. 
Such an $R$ is called a universal $R$-matrix of $A$. 

\par 
\begin{lem}[{\bf Drinfeld\cite{Dri}, Radford\cite{R}}]\label{properties_DE}
Let $(A,R)$ be a quasitriangular Hopf algebra. Then
\par 
$(1)$ the antipode $S$ is bijective,  
\par 
$(2)$ $R^{-1}=(S\otimes \text{id})(R)$, 
\par 
$(3)$ $R=(S\otimes S)(R)$, 
\par 
$(4)$ $(\varepsilon \otimes \text{id})(R)=1=(\text{id}\otimes \varepsilon )(R)$. 
\par 
Furthermore, if we write $R$ in the form 
$R=\sum R^{(1)}\otimes R^{(2)}$, 
and set $u:=\sum S(R^{(2)})R^{(1)}  \in A$, 
then the following conditions are satisfied. 
\begin{enumerate}
\itemindent=1cm 
\item[$(DE1)$]  $u$ is invertible, and $S^2(a)=uau^{-1}$ for all $a\in A$, 
\item[$(DE2)$] $\Delta (u)=(u\otimes u)(R_{21}R)^{-1}=(R_{21}R)^{-1}(u\otimes u)$,
\item[$(DE3)$] $\varepsilon (u)=1$,
\item[$(DE4)$] $u^{-1}=\sum  R^{(2)}S^2(R^{(1)})$. 
\end{enumerate}
Here, $R_{21}=\sum  R^{(2)}\otimes R^{(1)}$. 
The element $u$ is called the Drinfeld element of $(A, R)$. 
\qed 
\end{lem}

\par 
An element $v\in A$ is called a \textit{ribbon element} of a quasitriangular bialgebra $(A, R)$
 and the triplet $(A, R, v)$ is called a \textit{ribbon bialgebra} \cite{RT} if the following conditions are satisfied: 
\begin{enumerate}
\itemindent=1cm 
\item[(Rib1)] $v\in Z(A)$, where $Z(A)$ denotes the center of $A$, 
\item[(Rib2)] $\Delta (v)=(v\otimes v)(R_{21}R)^{-1}$,
\item[(Rib3)] $\varepsilon (v)=1$. 
\end{enumerate} 
In the case where $(A, R)$ is a quasitriangular Hopf algebra, the condition 
\begin{enumerate}
\itemindent=1cm 
\item[(Rib4)] $S(v)=v$ 
\end{enumerate} 
is also required in addition to the above three conditions.  
Then, the triplet $(A, R, v)$ is called a \textit{ribbon Hopf algebra}. 
By definition any ribbon element $v$ is invertible, and 
if $A$ is of finite dimension, then the condition
\begin{enumerate}
\itemindent=1cm 
\item[(Rib0)] $v^2=uS(u)$
\end{enumerate} 
is automatically satisfied \cite{Yetter}, where $u$ is the Drinfeld element of $(A, R)$. 

\par 
A ribbon element is characterized by a special group-like element as follows \cite{Kauff}. 

\par 
\begin{lem}\label{4-1}
Let $(A,R)$ be a quasitriangular Hopf algebra over $\boldsymbol{k}$. 
For an element $v\in A$ the following conditions (1) and (2) are equivalent. 
\par 
$(1)$ $v$ is a ribbon element of $(A,R)$. 
\par 
$(2)$ there is an element $g\in G(A)$ such that 
\par \smallskip \centerline{$\textrm{(a)}\ v=g^{-1}u,
\quad 
\textrm{(b)}\ S(u)=g^{-2}u, 
\quad 
\textrm{(c)}\ g^{-1}u\in Z(A). $}
\par \smallskip \noindent 
Here, $G(A)$ denotes the set of the group-like elements of $A$. \qed 
\end{lem}

\par 
Although the Drinfeld element is not necessary to be a ribbon element, 
 in the semisimple and cosemisimple case the following holds. 

\par 
\begin{prop}[{\bf Gelaki \cite[Lemma 2.1.1]{G}}] \label{4-rib}
Let $(A,R)$ be a quasitriangular Hopf algebra over $\boldsymbol{k}$, and $u$ be its Drinfeld element. 
If $A$ is semisimple and cosemisimple, then 
$u\in Z(A)$ and $u=S(u)$. 
Therefore, the Drinfeld element $u$ of 
$(A,R)$ is a ribbon element of $(A,R)$. \qed 
\end{prop}

\par 
Using Lemma~\ref{4-1} and Proposition~\ref{4-rib} we have: 

\par 
\begin{prop}\label{4-3}
Let $(A, R)$ be a finite-dimensional quasitriangular Hopf algebra over $\boldsymbol{k}$, and $u$ be its Drinfeld element. 
If $A$ is semisimple and cosemisimple, then 
the set of all ribbon elements $\text{Rib}(A,R)$ is given by 
$\text{Rib}(A,R)=\{ \ gu\ \vert \ g\in G(A)\cap Z(A),\ g^2=1\ \} $. \qed 
\end{prop}

\par 
In order to know the ribbon elements of a finite-dimensional semisimple and cosemisimple quasitriangular Hopf algebra $(A, R)$,  
it is enough to determine the set 
$\text{Sph}(A): =\{\ g\in G(A)\cap Z(A)\ \vert \ g^2=1\ \}$ by Proposition~\ref{4-3}. 
\par  
Let us recall the definitions of braided Hopf algebras and coribbon Hopf algebras that are the dual notions of quasitriangular Hopf algebras and ribbon Hopf algebras, respectively. 
The former and the letter are introduced by Doi \cite{Doi} and Hayashi \cite{Hayashi, Hayashi2}, respectively. 
Let $A$ be a bialgebra $A$ over $\boldsymbol{k}$. 
A linear functional $\sigma : A\otimes A\longrightarrow \boldsymbol{k}$ is called a \textit{braiding} of $A$, if it is convolution-invertible, 
and the following conditions are satisfied: 
\begin{enumerate}\itemindent=1cm 
\item[(B1)] $\sum\sigma (x_{(1)},y_{(1)})x_{(2)}y_{(2)}=\sum\sigma (x_{(2)},y_{(2)})y_{(1)}x_{(1)}$,
\item[(B2)] $\sigma (xy,z)=\sum\sigma (x,z_{(1)})\sigma (y,z_{(2)})$, 
\item[(B3)] $\sigma (x,yz)=\sum \sigma (x_{(1)},z)\sigma (x_{(2)},y)$
\end{enumerate}
for all $x,y,z\in A$. 
The pair $(A,\sigma )$ is called a \textit{braided bialgebra}. 
In a braided bialgebra $(A,\sigma )$ the following equation holds: 
\begin{enumerate}\itemindent=1cm 
\item[(B4)] $\sigma (1_A, x)=\sigma (x, 1_A)=\varepsilon (x)$ for all $x\in A$. 
\end{enumerate}

\par 
An invertible element $\theta \in A^{\ast }$ is said to be a \textit{coribbon element} of a braided bialgebra $(A,\sigma )$ if the following conditions are satisfied: 
\begin{enumerate}\itemindent=1cm 
\item[(CR1)] $\sum \theta (x_{(1)})x_{(2)}=\sum \theta (x_{(2)})x_{(1)}$,
\item[(CR2)] $\theta (xy)=\sum \sigma ^{-1}(x_{(1)},y_{(1)})\theta (x_{(2)})\theta (y_{(2)})\sigma ^{-1}(y_{(3)},x_{(3)})$, 
\item[(CR3)] $\theta (1)=1$ 
\end{enumerate} 
 for all $x,y\in A$. The triplet $(A,\sigma ,\theta )$ is called  a \textit{coribbon bialgebra}. 
Furthermore, if $A$ is a Hopf algebra and the condition
\begin{enumerate}\itemindent=1cm 
\item[(CR4)] $\theta \circ S=\theta $ 
\end{enumerate} 
is satisfied, then the triplet $(A,\sigma ,\theta )$ is called a \textit{coribbon Hopf algebra}. 

\par 
\begin{rem}
If a Hopf algebra $A$ is of finite dimension, then a braiding $\sigma $ of $A$ is a universal $R$-matrix of  $A^{\ast}$ via 
the usual isomorphism $(A\otimes A)^{\ast}\cong A^{\ast}\otimes A^{\ast}$. 
This construction gives a one-to-one correspondence between the  braidings of $A$ and the universal $R$-matrices of  $A^{\ast}$. 
Furthermore, 
an element $\theta \in A^{\ast}$ is a coribbon element of a braided Hopf algebra $(A, \sigma )$ if and only if 
it is a ribbon element of the quasitriangular Hopf algebra $(A^{\ast}, \sigma )$. 
\end{rem}

Dualizing Proposition~\ref{4-3} we have: 

\par 
\begin{cor}\label{4-4}
Let $(A, \sigma )$ be a finite-dimensional braided Hopf algebra over $\boldsymbol{k}$, and $\varUpsilon \in A^{\ast}$ be its Drinfeld element: 
\par \smallskip \centerline{$\varUpsilon (a)=\sum \sigma (a_{(2)}, S(a_{(1)}))\qquad (a\in A).$}
\par  \smallskip 
If $A$ is semisimple and cosemisimple, then the Drinfeld element $\varUpsilon $ is a coribbon element of $(A, \sigma )$, and the set of all coribbon elements of $(A, \sigma )$, written by $\text{CRib}(A,\sigma )$, is given by 
\par \smallskip \centerline{\hspace{1.5cm} $\text{CRib}(A,\sigma )=\{ \ p\varUpsilon   \ \vert \ p\in G(A^{\ast})\cap Z(A^{\ast}),\ p^2=\varepsilon \ \} .$\qed }
\end{cor}

\begin{rem} 
For the dual Hopf algebra  $A^{\ast}$, 
\begin{align*}
Z(A^{\ast})&=\{ \ p\in A^{\ast}\ \vert \ \forall \kern0.2em a \in A,\ \textstyle \sum p(a_{(1)})a_{(2)}=\sum p(a_{(2)})a_{(1)}\ \} , \\ 
G(A^{\ast})&=\{ \ p\in A^{\ast}\ \vert \ \forall \kern0.2em a,b\in A,\ p(ab)=p(a)p(b),\ p(1)=1\ \} .
\end{align*}
\end{rem}

\subsection{Polynomial invariants of quasitriangular Hopf algebras} 
In \cite{W} the author introduced some invariant of a finite-dimensional semisimple and cosemisimple Hopf algebra defined by using braiding structures and  given as a polynomial. 
This invariant is a monoidal Morita invariant for such a Hopf algebra. 
In this subsection we consider a braided refinement of the invariant. 
\par 
Let $A$ be a finite-dimensional semisimple and cosemisimple Hopf algebra over $\boldsymbol{k}$. 
By Etingof and Gelaki \cite[Corollary 1.5]{EG}, the set of universal $R$-matrices $\underline{\text{Braid}}(A)$ is finite. 
Let us consider a quasitriangular Hopf algebra  $(A,R)$.  
For an element $a\in A$ and a finite-dimensional left $A$-module $M$, let $\underline{a}_M: M\longrightarrow M$ denote the left action of  $a$ on $M$, and 
$u\in A$ is the Drinfeld element of $(A, R)$. Then, we set 
\par \smallskip \centerline{$\underline{\dim}_R\kern0.1em M=\text{Tr}(\underline{u}_M),$}
\par \smallskip \noindent 
and call it the categorical dimension of $M$ \cite{Majid}. 
We note that 
if $A$ is a finite-dimensional semisimple and cosemisimple Hopf algebra over  $\boldsymbol{k}$, then for any absolutely simple left $A$-module $M$, 
$(\dim M)1_{\boldsymbol{k}}\not= 0$ by \cite{EG}, and 
the following equation holds \cite[Lemma 3.2]{W}: 
\begin{equation}\label{eq4-1}
\Bigl( \dfrac{\underline{\text{dim}}_R M}{\dim M}\Bigr) ^{(\dim A)^3}=1. 
\end{equation}
So, $\underline{\text{dim}}_R M/\dim M$ is a root of unity in $\boldsymbol{k}$. 

\par 
Let $d$ be a positive integer, and $\{ M_1, \ldots , M_t\}$ be a complete system of the absolutely simple left $A$-modules of dimension $d$. 
Then we define a polynomial $P_{A,R}^{(d)}(x)\in \boldsymbol{k}[x]$ by 
\begin{equation}
P_{A,R}^{(d)}(x)=\prod\limits_{i=1}^t \Bigl( x-\dfrac{\underline{\dim}_RM_i}{d}\Bigr) . 
\end{equation}
If there is no absolutely simple left $A$-module of dimension $d$, then we define $P_{A,R}^{(d)}(x):=1$. 
\par 
For a quasitriangular Hopf algebra $(A,R)$  we denote the braided monoidal category  $({}_A\mathbb{M}, c^R)$ by ${}_{(A,R)}\mathbb{M}$. 
Here, $c^R$ is the braiding associated to $R=\sum R^{(1)}\otimes R^{(2)}$, that is, for $M, N\in {}_A\mathbb{M}$ 
\par \smallskip \centerline{$(c^R)_{M, N}(m\otimes n)=\sum (R^{(2)}\cdot n)\otimes (R^{(1)}\cdot m)\qquad (m\in M,\ n\in N).$}
\par \smallskip 
Two quasitriangular Hopf algebras $(A,R)$ and $(B, R^{\prime})$ over $\boldsymbol{k}$ are said to be 
\textit{braided Morita equivalent} if the braided monoidal categories 
${}_{(A,R)}\mathbb{M}$ and ${}_{(B,R^{\prime})}\mathbb{M}$ are equivalent as 
$\boldsymbol{k}$-linear braided monoidal categories. 
By using the same technique in the proof of \cite[Theorem 2.6]{W} 
it can be verified that the above polynomial $P_{A,R}^{(d)}(x)$ is a braided Morita invariant, that is, 
$(A,R)$ and $(B, R^{\prime})$ are braided Morita equivalent, then 
$P_{(A,R)}^{(d)}(x)=P_{(B, R^{\prime})}^{(d)}(x)$ for all positive integers $d$. 
By using $P_{A,R}^{(d)}(x)$, 
the polynomial invariant $P_A^{(d)}(x)$ defined in \cite{W} can be written by 
$P_A^{(d)}(x)=\prod_{R\in \underline{\text{Braid}}(A)}P_{A,R}^{(d)}(x)$. 
\par 
Another braided Morita invariant can be constructed by using ribbon structures. 
Let $(A, R)$ be a quasitriangular Hopf algebra, and $v\in A$ be its ribbon element. 
Then any ribbon element $v$ of $(A,R)$ induces a twist $\theta ^v = \{ \theta ^v_M\}_{M\in {}_A\mathbb{M}}$ for the braided category ${}_{(A,R)}\mathbb{M}$, 
where $\theta ^v_M: M\longrightarrow M$ is an $A$-linear isomorphism defined by
\par \smallskip \centerline{$(\theta ^v_M)(m)=v^{-1}\cdot m\qquad (m\in M).$}
\par \smallskip 
Suppose that $A$ is finite-dimensional semisimple and cosemisimple. 
For a positive integer $d$ a polynomial $\tilde{P}_{A,R}^{(d)}(x)$ can be defined as follows. 
\begin{equation}
\tilde{P}_{A,R}^{(d)}(x):=\prod\limits_{v\in \text{Rib}(A,R)}\prod\limits_{i=1}^t (x-\xi _v(M_i)), 
\end{equation}
where $\{ M_1, \ldots , M_t\} $ is a complete system of the absolutely simple left $A$-modules of dimension $d$, 
and  $\xi _v(M_i)$ is a scalar determined by $\underline{v}_{M_i}=\xi _v(M_i)\text{id}_{M_i}$. 
This polynomial $\tilde{P}_{A,R}^{(d)}(x)$ is also a braided Morita invariant. 
By Proposition~\ref{4-rib} and Lemma~\ref{3-10} given in the next subsection we have: 

\par 
\begin{prop} \label{4-11} 
Let $(A,R)$ be a semisimple and cosemisimple quasitriangular Hopf algebra of finite dimension. 
For a positive integer $d$, 
$\tilde{P}_{A,R}^{(d)}(x)$ can be divided by $P_{A,R}^{(d)}(x)$ in $\boldsymbol{k}[x]$. 
So, a polynomial $Q_{A,R}^{(d)}(x)=\tilde{P}_{A,R}^{(d)}(x)/P_{A,R}^{(d)}(x)$ 
$\in \boldsymbol{k}[x]$
is defined, and it is also a braided Morita invariant. 
\end{prop}

\par 
\begin{exam}
Let $C_n$ denote the cyclic group of order $n$ which is generated by $a$, and 
$\omega \in \mathbb{C}$ be a primitive $n$th root of unity. 
The universal $R$-matrices of the group Hopf algebra $\mathbb{C}C_n$ are  
\par \smallskip \centerline{$R_d=\sum\limits_{k,l=0}^{n-1}\omega ^{dkl}E_k\otimes E_l\quad (d=0,1,\ldots , n-1),$}
\par \smallskip \noindent 
where $E_k=\frac{1}{n}\sum _{i=0}^{n-1}\omega ^{-ik}a^i$ for each $k\in \mathbb{Z}$. 
The Drinfeld element $u_d$ of $(\mathbb{C}C_n, R_d)$  is 
$u_d=\sum_{k=0}^{n-1}\omega ^{-dk^2}E_k$. 
We have 
\par \smallskip \centerline{  
$\text{Rib}(\mathbb{C}C_n, R_d)=\begin{cases}
\{ u_d\} & \text{if $n$ is odd}, \\ 
\{ u_d, u_da^{\frac{n}{2}}\} & \text{if $n$ is even}. 
\end{cases}
$}
\par \smallskip \noindent 
If we set $M_k=\mathbb{C}E_k$, then $\{ M_0, M_1, \ldots , M_{n-1}\} $ 
forms a complete system of simple $\mathbb{C}C_n$-modules. 
Then $\xi _{u_d}(M_k)=\underline{\dim}_{R_d} M_k=\text{Tr}(\underline{u_d}_{M_k})=\omega ^{-dk^2}$, and 
if $n$ is even, then $\xi _{\overline{u}_d}(M_k)=(-1)^k\underline{\dim}_{R_d} M_k=(-1)^k\omega ^{-dk^2}$ for $\overline{u}_d:=u_da^{\frac{n}{2}}$ by using 
$a^jE_k=\omega ^{jk}E_k\ (j,k\in \mathbb{Z})$. 
Therefore 
\begin{align*}
P_{(\mathbb{C}C_n, R_d)}^{(1)}(x)
&=\prod\limits_{k=0}^{n-1}(x-\omega ^{-dk^2}),\\ 
Q_{(\mathbb{C}C_n, R_d)}^{(1)}(x)
&=\begin{cases}
1 & \text{if $n$ is odd}, \\ 
\prod\limits_{k=0}^{n-1}(x-(-1)^k\omega ^{-dk^2}) & \text{if $n$ is even}. 
\end{cases} 
\end{align*}
By comparing $P_{(\mathbb{C}C_n, R_d)}^{(1)}(x)$ we see that $(\mathbb{C}C_n, R_d)\ (d=0,1,\ldots , n-1)$ are not mutually braided Morita equivalent  for $n=2,3,4$. 
In the case of $n=5$
\begin{align*}
P_{(\mathbb{C}C_5, R_0)}^{(1)}(x)&=(x-1)^5,\displaybreak[0]\\ 
P_{(\mathbb{C}C_5, R_1)}^{(1)}(x)&=P_{(\mathbb{C}C_5, R_4)}^{(1)}(x)=(x-1)(x-\omega ^{-1})^2(x-\omega )^2,\displaybreak[0]\\ 
P_{(\mathbb{C}C_5, R_2)}^{(1)}(x)&=P_{(\mathbb{C}C_5, R_3)}^{(1)}(x)=(x-1)(x-\omega^3)^2(x-\omega^2)^2. 
\end{align*}
So, $(\mathbb{C}C_5, R_d)\ (d=0,1,2)$ are not mutually braided Morita equivalent. 
\end{exam} 

\subsection{Relationship between ribbon and pivotal structures}  
It is known that any ribbon category has a pivotal structure \cite{Yetter}. 
In the case where a ribbon category is the module category ${}_A\mathbb{M}^{\text{fd.}}$ of finite-dimensional left $A$-modules over a ribbon Hopf algebra $(A, R, v)$,  
the associated pivotal structure $\tau$ is given by 
\par \smallskip \centerline{$\tau_M(m)=uv^{-1}\cdot m\quad (m\in M)$}
\par \smallskip \noindent 
for each object $M\in {}_A\mathbb{M}^{\text{fd.}}$. 
Therefore the left and right pivotal dimensions of $M$ in the ribbon category are 
\begin{equation}\label{eq1-7}
\underline{\text{pdim}}_{uv^{-1}}^{\ell} M =\underline{\text{pdim}}_{uv^{-1}}^{r} M =\text{Tr}(\underline{uv^{-1}}_M). 
\end{equation}

Suppose that $M$ is absolutely simple. Since $v$ is invertible, it follows that $\xi_v(M)\not= 0$, and hence 
$$
\underline{\text{pdim}}_{uv^{-1}}^{\ell} M =\underline{\text{pdim}}_{uv^{-1}}^{r} M 
=\xi_v(M)^{-1}\text{Tr}(\underline{u}_M)=\xi_v(M)^{-1}\underline{\dim}_R M. 
$$
From this, we also have 
\begin{equation}\label{eq1-9}
\dfrac{\underline{\text{pdim}}_{uv^{-1}}^{r} M }{\dim M}
=\xi_v(M)^{-1}\dfrac{\underline{\dim}_R M}{\dim M}. 
\end{equation}

If $A$ is semisimple and cosemisimple, then the pivotal structures of ${}_A\mathbb{M}^{\text{fd.}}$ are uniquely determined by the group $Z(A)\cap G(A)$. 
Thus, $\eta :=uv^{-1}$ has finite order, and it follows that $\underline{\text{pdim}}_{uv^{-1}}^{r} M /\dim M$ is a root of unity in $\boldsymbol{k}$. 
By \eqref{eq4-1}, $\underline{\dim}_R M/\dim M$ is also a root of unity, and so by \eqref{eq1-9}, $\xi_v(M)^{-1}$ is, too. 

\par 
\begin{lem}\label{3-10} 
Let $(A, R)$ be a semisimple and cosemisimple quasitriangular Hopf algebra of finite dimension, 
and $M$ be an absolutely simple left $A$-module. Then 
$\xi_u(M)=\frac{\underline{\dim}_RM}{\dim M}$ for the Drinfeld element $u$ of $(A, R)$. 
\end{lem}
\begin{proof}
The pivotal element  $\eta $ corresponding to $u$ is $\eta =uu^{-1}=1_A$. 
Since 
$\underline{\text{pdim}}_{\eta }^r M$ coincides with the trace of $\underline{\eta}_{M}$ by \eqref{eq1-7},  
we see that $\underline{\text{pdim}}_{\eta }^r M=(\dim M)1_{\boldsymbol{k}}$. 
Now, the desired equation follows from \eqref{eq1-9}. 
\end{proof}

\par \medskip 
Let $C$ be a coalgebra over $\boldsymbol{k}$. 
The dual space $C^{\ast}$ has a $\boldsymbol{k}$-algebra structure, and 
any finite-dimensional right $C$-comodule $M$ can be regarded as a left $C^{\ast}$-module with 
the action 
\par \smallskip \centerline{$p\cdot m=\sum p(m_{(1)})m_{(0)} \qquad (p\in C^{\ast},\ m\in M),$}
\par \smallskip \noindent 
where we write the right $C$-coaction $\rho: M\longrightarrow M\otimes C$ in the form $\rho (m)=\sum m_{(0)}\otimes m_{(1)}$. 
This construction gives rise to an identical category equivalence between $\boldsymbol{k}$-linear monoidal categories of finite-dimensional right $C$-comodules and of finite-dimensional left $C^{\ast}$-modules. 
For a finite-dimensional right $C$-comodule $M$, 
an element $\text{ch}(M)\in C$ called the \textit{character} of $M$ is defined by 
\par \smallskip \centerline{$\text{ch}(M):=\sum\limits_{i=1}^d(e_i^{\ast}\otimes \text{id}_C)(\rho(e_i)),$}
\par \smallskip \noindent 
where $\{ e_i\}_{i=1}^d, \{ e_i^{\ast}\}_{i=1}^d$ are mutually dual bases of $M, M^{\ast}$, respectively. 
The above element does not depend on the choice of bases. 

\par 
\begin{lem}\label{4-2}
Let $A$ be a finite-dimensional Hopf algebra over $\boldsymbol{k}$, and $M$ be a finite-dimensional right $A$-comodule. 
\par 
$(1)$ $\eta \in A^{\ast}$ is a pivotal element of the dual Hopf algebra $A^{\ast}$, then 
$\underline{\text{pdim}}_{\eta }^r\kern0.1em M=\eta (\text{ch}(M))$, 
where the left-hand side is the right pivotal dimension of $M$ viewed as a left  $A^{\ast}$-module as usual. 
\par 
$(2)$ Assume that $M$ is absolutely simple with $(\dim M)1_{\boldsymbol{k}}\not= 0$. 
Then $\xi_p(M)=\frac{p(\text{ch}(M))}{\dim M}$ for any $p\in Z(A^{\ast})$. 
\end{lem}
\begin{proof}
Let $\{ e_i\}_{i=1}^d$ and $\{ e_i^{\ast}\}_{i=1}^d$ be mutually dual bases of $M$ and $M^{\ast}$, respectively, and 
$\rho $ be the right coaction on $A$. 
We write $\rho (e_i)=\sum_{j=1}^d e_j\otimes a_{ji}\ (a_{ji}\in A)$. 
Then $\text{ch}(M)=\sum_{i=1}^d a_{ii}$, and hence 
$p(\text{ch}(M))=\sum_{i=1}^d p(a_{ii})$. 
On the other hand, 
$\underline{p}_M(e_i)=\sum_{j=1}^d p(a_{ji}) e_j$. Thus we have 
$\text{Tr}(\underline{p}_M)=\sum_{j=1}^d p(a_{ii})=p(\chi _M)$. 
Since $\text{Tr}(\underline{\eta}_M)=\underline{\text{pdim}}_{\eta}^r\kern0.1em M$ for a pivotal element $\eta$, Part (1) is proved.  
Assume that $M$ is absolutely simple with $(\dim M)1_{\boldsymbol{k}}\not= 0$. 
Then there is an element $\xi _p(M)\in \boldsymbol{k}$ such that $\underline{p}_M=\xi _p(M) \text{id}_M$. 
Taking the trace of this map we have the formula in Part (2). 
\end{proof}

\section{The coribbon elements of Suzuki's Hopf algebras}
In this section we review the definition of Suzuki's Hopf algebras, and describe the braiding structures of them in accordance with Suzuki's paper \cite{Suzuki2}. 
The correct results on coribbon elements of Suzuki's braided Hopf algebras described in \cite{W2} are also given. 
\par 
Suzuki's Hopf algebras are given as a family of  finite-dimensional cosemisimple Hopf algebras generated by  a comatrix basis of the $2\times 2$-matrices. 
Suppose that $\boldsymbol{k}$ is an algebraically closed field whose characteristic is not $2$, and let 
$C=(C, \Delta , \varepsilon )$ be the comatrix coalgebra of degree $2$ over  $\boldsymbol{k}$, that is,  
there is a basis $\{ X_{11},X_{12},X_{21},X_{22}\} $ of $C$ such that 
$$\Delta (X_{ij})=X_{i1}\otimes X_{1j}+X_{i2}\otimes X_{2j},\qquad \varepsilon (X_{ij})=\delta _{ij}.$$
Let $I$ be a coideal of the tensor algebra  $\mathcal{T}(C)$ defined by 
$$I=\boldsymbol{k}(X_{11}^2-X_{22}^2)+\boldsymbol{k}(X_{12}^2-X_{21}^2)+\sum\limits_{i-j\not\equiv l-m\ (\text{mod}\kern0.2em 2)}\boldsymbol{k}(X_{ij}X_{lm}).$$
We set $B:=\mathcal{T}(C)/\langle I\rangle $, and denote by $x_{ij}$ the image of $X_{ij}$ under the natural projection  $\mathcal{T}(C)\longrightarrow B$. 
For $i,j=1,2,\ m\geq 1$ we define an element $\chi _{ij}^m$ in $B$ by 
\begin{equation}\label{eq2-1}
\begin{aligned}
\chi _{11}^m & :=\underbrace{x_{11}x_{22}x_{11}\cdots\cdots }_{\text{$m$}}\ , \quad & 
\chi _{22}^m & :=\underbrace{x_{22}x_{11}x_{22}\cdots\cdots }_{\text{$m$}}, \\ 
\chi _{12}^m & :=\underbrace{x_{12}x_{21}x_{12}\cdots\cdots }_{\text{$m$}}\ ,\quad & 
\chi_{21}^m & :=\underbrace{x_{21}x_{12}x_{21}\cdots\cdots }_{\text{$m$}}.
\end{aligned}
\end{equation}
Then we have $\Delta (\chi _{ij}^m)=\chi _{i1}^m\otimes \chi _{1j}^m+\chi _{i2}^m\otimes \chi _{2j}^m$. 
\par 
Let $N\geq 1,\ L\geq 2$, $\nu ,\lambda =\pm 1$, and consider the following subset of $B$: 
\begin{equation}
J_{NL}^{\nu \lambda }:=\boldsymbol{k}(x_{11}^{2N}+\nu x_{12}^{2N}-1) \notag \\ 
+\boldsymbol{k}(\chi _{11}^L-\chi_{22}^L)+\boldsymbol{k}(-\lambda \chi_{12}^L+\chi_{21}^L). 
\end{equation}
Then $J_{NL}^{\nu \lambda }$ is a coideal of $B$, and 
$A_{NL}^{\nu \lambda }:=B/\langle J_{NL}^{\nu \lambda }\rangle $
is a bialgebra. 
We also denote the image of $x_{ij}$ by the same symbol. 
It can be easily shown that 
\begin{equation}\label{eq_basis_of_A}
\{ \ x_{11}^s\chi_{22}^t,\ x_{12}^s\chi_{21}^t\ \vert \ 1\leq s\leq 2N,\ 0\leq t\leq L-1\ \}
\end{equation}
is a basis of $A_{NL}^{\nu \lambda }$ over $\boldsymbol{k}$, and $\dim A_{NL}^{\nu \lambda }=4NL$. 
The bialgebra $A_{NL}^{\nu \lambda }$ actually is  a cosemisimple Hopf algebra, whose structure maps are given by 
\par \smallskip \centerline{$\Delta (x_{ij})=x_{i1}\otimes x_{1j}+x_{i2}\otimes x_{2j},\quad \varepsilon (x_{ij})=\delta _{ij},\quad S(x_{ij})=x_{ji}^{4N-1}.$}

\begin{rem}
 The description of the antipode of $A_{Nn}^{\nu \lambda}$ in \cite{W2} is wrong in the case when $\nu =-$.
\end{rem}

\par 
If  $\text{ch}(\boldsymbol{k})\nmid NL$, then the cosemisimple Hopf algebra  $A_{NL}^{\nu \lambda}$ is also semisimple \cite[Theorem 3.1 vii)]{Suzuki2}, and 
$A_{1L}^{++}, A_{1L}^{+-}$ coincide with $\mathcal{A}_{4L}, \mathcal{B}_{4L}$, respectively, which are introduced by Masuoka~\cite{Masuoka1} and generalized in \cite{CDMM}. 
In particular, the Hopf algebra $H_8=A_{12}^{+-}$ is the unique Hopf algebra which is an $8$-dimensional non-commutative and non-cocommutative Hopf algebra up to isomorphism. 
This Hopf algebra is called the Kac-Paljutkin algebra \cite{KP, Masuoka0}. 
By uniqueness we see that $H_8$ is self-dual, that is the dual Hopf algebra is isomorphic to itself. 

\par 
By \eqref{eq_basis_of_A}, we see that for any integers  $s,t\geq 0$ satisfying with $s+t\geq 1$, 
\begin{align*}
\Delta (x_{11}^s\chi _{22}^t)&=x_{11}^s\chi _{22}^t\otimes x_{11}^s\chi _{22}^t+x_{12}^s\chi _{21}^t\otimes x_{21}^s\chi _{12}^t,\\ 
\Delta (x_{12}^s\chi _{21}^t)&=x_{11}^s\chi _{22}^t\otimes x_{12}^s\chi _{21}^t+x_{12}^s\chi _{21}^t\otimes x_{22}^s\chi _{11}^t, \displaybreak[0]\\ 
S(x_{11}^s\chi _{22}^t)&=\begin{cases}
x_{22}^{(4N-2)(t+s)+s}\chi _{11}^t & \text{if $s+t$ is even},\\ 
x_{11}^{(4N-2)(t+s)+s}\chi _{22}^t & \text{if $s+t$ is odd}, \end{cases} \displaybreak[0]\\ 
S(x_{12}^s\chi _{21}^t)&=\begin{cases}
x_{12}^{(4N-2)(t+s)+s}\chi _{21}^t &\text{if $s+t$ is even},\\ 
x_{21}^{(4N-2)(t+s)+s}\chi _{12}^t &\text{if $s+t$ is odd}. 
\end{cases}
\end{align*}

It is known by Suzuki \cite{Suzuki2} that the group $G(A_{NL}^{\nu \lambda })$ is given by 
\begin{equation}\label{eq_grouplikes_of_SuzukiHopf}
G(A_{NL}^{\nu \lambda })=\{ \ 
x_{11}^{2s}\pm x_{12}^{2s},\ 
x_{11}^{2s+1}\chi _{22}^{L-1}\pm \sqrt{\lambda}x_{12}^{2s+1}\chi_{21}^{L-1}\ \vert \ 1\leq s\leq N\ \},
\end{equation}
\noindent 
that is of order $4N$, and  the set 
$$\{ \ \boldsymbol{k}g\ \vert \ g\in G(A_{NL}^{\nu \lambda })\ \} 
\cup \{ \ \boldsymbol{k}x_{11}^{2s}\chi_{22}^t+\boldsymbol{k}x_{12}^{2s}\chi_{21}^t  \ \vert \ 0\leq s\leq N-1,\ 1\leq t\leq L-1\ \} $$
gives a complete system of absolutely simple right $A_{NL}^{\nu \lambda}$-comodules, 
where the coactions of all subspaces above are induced from the comultiplication $\Delta $ of $A_{NL}^{\nu \lambda }$. 
\par 
Let $N$ be odd, and set $\lambda =+$ or $-$ if $L$ is odd or even, respectively. 
Then, $A_{NL}^{\nu \lambda }$ is isomorphic to the group algebra of the following finite group \cite{W}: 
\par \smallskip \centerline{$G_{NL}=\langle h,\ t,\ w\ \vert \ t^2=h^{2N}=1,\ w^L=h^{N},\ tw=w^{-1}t,\ ht=th,\ hw=wh \rangle .$}
\par \smallskip \noindent 
In fact, an algebra isomorphism 
$\varphi :\boldsymbol{k}[G_{NL}]\longrightarrow A_{NL}^{\nu \lambda }$ is given by 
\par \smallskip \centerline{$\varphi (h)=x_{11}^2-x_{12}^2,\quad \varphi (t)=x_{12}^N+x_{22}^N, \quad 
\varphi (w)=x_{11}^{2N-1}x_{22}-x_{21}^{2N-1}x_{12}. $}

\par \smallskip 
Suzuki \cite{Suzuki2} also determined the all braidings of $A_{NL}^{\nu \lambda }$. 
The construction of $A_{NL}^{\nu \lambda }$ and the method of determination of its braidings are closely related to the universality for quadratic bialgebras (see \cite{Doi} for a detailed statement and also \cite{W2} for the above fact). 

\begin{center}
\begin{tabular}{c|cccc}
$x\diagdown y$ & $x_{11}$ & $x_{12}$ & $x_{21}$ & $x_{22}$ \\
\hline 
 $x_{11}$  & $0$ & $0$ & $0$ & $0$ \\
 $x_{12}$  & $0$ & $\alpha $ & $\beta $ & $0$ \\
 $x_{21}$  & $0$ & $\beta $ & $\alpha $ & $0$ \\
 $x_{22}$  & $0$ & $0$ & $0$ & $0$ \\ 
\end{tabular} \hspace{2cm}  
\begin{tabular}{c|cccc}
$x\diagdown y$ & $x_{11}$ & $x_{12}$ & $x_{21}$ & $x_{22}$ \\
\hline 
 $x_{11}$  & $\gamma $ & $0$ & $0$ & $\delta $ \\
 $x_{12}$  & $0$ & $0$ & $0$ & $0$ \\
 $x_{21}$  & $0$ & $0$ & $0$ & $0$ \\
 $x_{22}$  & $\lambda \delta $ & $0$ & $0$ & $\gamma $ \\ 
\end{tabular}
\end{center}

\par 
\begin{thm}[{\bf S.Suzuki} {\cite{Suzuki2}}]\label{Suzuki_braiding}\label{1-2}  
$(1)$  For $\alpha , \beta \in \boldsymbol{k}$, 
let $\sigma _{\alpha \beta }: C\otimes C\longrightarrow \boldsymbol{k}$ be a $\boldsymbol{k}$-linear map whose values $\sigma _{\alpha \beta }(x_{ij}, x_{kl})\ (i,j,k,l=1,2)$ are given by the left table above. 
Then $\sigma _{\alpha \beta}$ is extended to a braiding of $A_{NL}^{\nu \lambda }$ if and only if 
$\alpha , \beta \in \boldsymbol{k}^{\times},\ (\alpha \beta )^N=\nu ,\ (\alpha \beta ^{-1})^L=\lambda$. 
\par 
\indent 
$(2)$ Consider the case $L=2$. For  
$\gamma , \delta \in \boldsymbol{k}$,  
let  $\tau _{\gamma \delta }^{\lambda}: C\otimes C\longrightarrow \boldsymbol{k}$ be a $\boldsymbol{k}$-linear map whose values $\tau _{\gamma \delta }^{\lambda }(x_{ij}, x_{kl})\ (i,j,k,l=1,2)$ are given by the right table above. 
Then, $\tau _{\gamma \delta }^{\lambda }$ is extended to a braiding of $A_{N2}^{\nu \lambda }$  if and only if 
$\gamma , \delta \in \boldsymbol{k}^{\times},\ \gamma ^2=\delta ^2,\ \gamma ^{2N}=1$. 
\par 
$(3)$ If $L\geq 3$, then the braidings of $A_{NL}^{\nu \lambda }$ are given by 
\par \smallskip \centerline{$\{ \ \sigma _{\alpha \beta }\ \vert \ \alpha , \beta \in \boldsymbol{k}^{\times},\ (\alpha \beta )^N=\nu ,\ (\alpha \beta ^{-1})^L=\lambda \ \} .$}
\par \smallskip 
If $L=2$, then  the braidings of $A_{N2}^{\nu \lambda }$ are given by 
\begin{align*}
&\{ \ \sigma _{\alpha \beta }\ \vert \ \alpha , \beta \in \boldsymbol{k}^{\times},\ (\alpha \beta )^N=\nu ,\ (\alpha \beta ^{-1})^2=\lambda \ \} \\ 
&\hspace{2cm} \cup \{\ \tau _{\gamma \delta }^{\lambda }\ \vert \ \gamma , \delta \in \boldsymbol{k}^{\times},\ \gamma ^2=\delta ^2,\ \gamma ^{2N}=1\ \} .
\end{align*}
\end{thm}

\par 
We note that there is a natural embedding $C\subset A_{NL}^{\nu \lambda }$, and therefore $A_{NL}^{\nu \lambda }$ is generated by $C$ as an algebra. 
Thus, a braiding $\sigma $ of $A_{NL}^{\nu \lambda }$ is determined by the values on $C$ by  (B2), (B3). 

\par 
The following lemma is partially proved in \cite[p.341]{W2}. 
The equation  
$\sigma _{\alpha \beta }^{-1}(x_{21}^{m-1}, x_{22})
=\sigma _{\alpha \beta }^{-1}(x_{22}, x_{12}^{m-1})
=\alpha ^{-\frac{m-1}{2}}\beta ^{-\frac{m-1}{2}}$ 
for an odd integer $m\geq 3$ is added. 
In particular, these values are not equal to $0$. 
Hereinafter, we treat the indices of Kronecker's delta $\delta _{ij}$ as modulo  $2$. 

\par 
\begin{lem}\label{2-6}
In the braided Hopf algebra $(A_{NL}^{\nu \lambda }, \sigma _{\alpha \beta })$ 
the following holds. 
\begin{align*}
\sigma _{\alpha \beta }^{\pm 1}(x_{ij}^m, x_{kl})
&=\sigma _{\alpha \beta }^{\pm 1}(x_{kl}, x_{ij}^m) \\ 
&=\begin{cases}
\delta _{i+j,1}\delta _{k+l, 1}(\alpha ^{\pm 1})^{\frac{m-1}{2}+\delta _{i,k}}(\beta^{\pm 1})^{\frac{m-1}{2}+\delta _{j,k}} & \text{if $m$ is odd},\\ 
\delta _{i+j,1}\delta _{k,l}(\alpha ^{\pm 1})^{\frac{m}{2}}(\beta^{\pm 1})^{\frac{m}{2}} & \text{if $m$ is even}. 
\end{cases}
\end{align*}
\end{lem}

\par 
\begin{lem}\label{2-10}
In the braided Hopf algebra $(A_{N2}^{\nu \lambda }, \tau _{\gamma \delta }^{\lambda })$ the following holds:  
$$(\tau _{\gamma \delta }^{\lambda })^{\pm 1}(x_{ij}^m, x_{kl})
=(\tau _{\gamma \delta }^{\lambda })^{\pm 1}(x_{ij}, x_{kl}^m)
=\begin{cases}
\gamma ^{\pm m} & \text{if $i=j=k=l$},\\ 
\delta ^{\pm m} & \text{if $i=j=1,\ k=l=2$},\\ 
(\lambda \delta )^{\pm m} & \text{if $i=j=2,\ k=l=1$},\\ 
0 & \text{otherwise}. 
\end{cases}
$$
\end{lem}

\par 
The Drinfeld elements of Suzuki's braided Hopf algebras are given by the following lemma. 

\par  
\begin{lem}\label{4-9}
Suppose that $\boldsymbol{k}$ contains a $4NL$th root of unity. 
\par 
$(1)$ The Drinfeld element $\varUpsilon _{\alpha \beta}$ of $(A_{NL}^{\nu \lambda}, \sigma _{\alpha \beta })$ is given by 
$\varUpsilon _{\alpha \beta}(x_{ij})=\delta _{i,j}\beta ^{-1}$. 
\par 
$(2)$ The Drinfeld element $\varUpsilon _{\gamma \delta }^{\lambda}$ of $(A_{N2}^{\nu \lambda}, \tau _{\gamma \delta }^{\lambda})$ is given by $\varUpsilon _{\gamma \delta }^{\lambda}(x_{ij})=\delta _{i,j}\gamma ^{-1}$. 
\end{lem}
\begin{proof}
(1) $\varUpsilon _{\alpha \beta}(x_{ij})=\sigma _{\alpha \beta }(x_{1j}, S(x_{i1}))
+\sigma _{\alpha \beta }(x_{2j}, S(x_{i2}))=\sigma _{\alpha \beta }(x_{1j}, x_{1i}^{4N-1})$ 
$+\sigma _{\alpha \beta }(x_{2j}, x_{2i}^{4N-1}).$ 
Since $\sigma _{\alpha \beta }(x_{1j}, x_{1i}^{4N-1})=\delta _{j, 0}\delta _{i, 0}\beta ^{-1},\ 
\sigma _{\alpha \beta }(x_{2j}, x_{2i}^{4N-1})=\delta _{j, 1}\delta _{i, 1}\beta ^{-1}$ by Lemma~\ref{2-6}, 
it follows that $\varUpsilon _{\alpha \beta}(x_{ij})=\delta _{i,j}\beta ^{-1}$. 
\par 
(2) By Lemma~\ref{2-10} and $\gamma ^{2N}=1$, we have 
$ \varUpsilon _{\gamma \delta }^{\lambda}(x_{ij})
=\tau _{\gamma \delta }^{\lambda}(x_{1j}, x_{1i}^{4N-1})+\tau _{\gamma \delta }^{\lambda}(x_{2j}, x_{2i}^{4N-1}) 
=\delta _{i1}\delta _{j1}\gamma^{-1}+\delta _{i2}\delta _{j2}\gamma^{-1}
=\delta _{ij}\gamma^{-1}$. 
\end{proof}

\par \medskip 
The following is the revised version of  Lemma 8 in \cite{W2} (see Appendix for the needed modification). 

\par 
\begin{lem}\label{2-9}
The Yang-Baxter form $\sigma _{\alpha \beta }$ on $C$ given in Theorem~\ref{Suzuki_braiding} (1) can be extended to a braiding of the bialgebra $B=\mathcal{T}(C)/\langle I\rangle $. 
We denote it the same symbol $\sigma _{\alpha \beta }$. 
For an element $\omega \in \boldsymbol{k}^{\times }$, the $\boldsymbol{k}$-linear functional $\theta _{\omega }:C\longrightarrow \boldsymbol{k}$ defined by 
$\theta _{\omega }(x_{ij})=\delta _{ij}\omega \ (i,j=1,2)$ 
 can be extended to a coribbon element of the braided bialgebra $(B, \sigma _{\alpha \beta })$. 
We denote the coribbon element by the same symbol $\theta _{\omega }$. 
Suppose that $\alpha $ and $\beta $ satisfy  $(\alpha \beta )^N=\nu ,\ (\alpha \beta^{-1})^L=\lambda $. Then, 
\par 
$(1)$ $\theta _{\omega }$ induces a coribbon element of the braided bialgebra $(A_{NL}^{\nu \lambda } , \sigma _{\alpha \beta })$ if and only if $\omega ^{2N}=\alpha ^{2N}$. 
\par 
$(2)$ for $\omega $ with $\omega ^{2N}=\alpha ^{2N}$, 
$\theta _{\omega }\circ S=\theta _{\omega }$ if and only if $\omega =\pm \beta ^{-1}$. \qed 
\end{lem}

\par 
\begin{lem}\label{2-12}
The Yang-Baxter form $\tau _{\gamma \delta }^{\lambda }$ on $C$ given in Theorem~\ref{Suzuki_braiding} (2) can be extended to a braiding of the bialgebra $B^{(\lambda )} =\mathcal{T}(C)/\langle I^{(\lambda )} \rangle $, where 
\par \medskip \centerline{$I^{(\lambda )}=\boldsymbol{k}(X_{11}X_{22}-X_{22}X_{11})
+\boldsymbol{k}(X_{12}X_{21}-\lambda X_{21}X_{12})
+\kern-1em \sum\limits_{i-j\not\equiv l-m\kern0.2em (\text{mod}\kern0.2em 2)}\kern-1em \boldsymbol{k}(X_{ij}X_{lm})$.}
\par \noindent 
We denote this braiding of $B^{(\lambda )} $ by the same symbol $\tau _{\gamma \delta }^{\lambda }$. 
For an element  $\omega \in \boldsymbol{k}^{\times } $, the $\boldsymbol{k}$-linear functional  $\theta _{\omega }:C\longrightarrow \boldsymbol{k}$ by the same formula in Lemma~\ref{2-9} can be extended to a coribbon element of the braided bialgebra $(B^{(\lambda )} ,\tau _{\gamma \delta }^{\lambda })$. 
We denote the coribbon element by the same symbol $\theta _{\omega }$. 
Suppose that $\gamma $ and $\delta $ satisfy $\gamma ^2=\delta ^2,\ \gamma ^{2N}=1$. Then, 
\par 
$(1)$ $\theta _{\omega }$ induces a coribbon element of the braided bialgebra $(A_{N2}^{\nu \lambda } , \tau _{\gamma \delta }^{\lambda })$ if and only if $\omega ^{2N}=1$. 
\par 
$(2)$ for $\omega \in \boldsymbol{k}$ with $\omega ^{2N}=1$,  
$\theta _{\omega }\circ S=\theta _{\omega }$ if and only if $\omega =\pm \gamma ^{-1}$. \qed 
\end{lem}

\par \smallskip 
Combining Lemmas~\ref{2-9} and \ref{2-12}, 
we have the following correct version of \cite[Theorem 5]{W2}. 

\par 
\begin{thm}\label{2-2} 
Let $\boldsymbol{k}$ be an algebraically closed field whose characteristic does not divide $2NL$. 
For each element $\omega \in \boldsymbol{k}$, 
let $\theta _{\omega }: C\longrightarrow \boldsymbol{k}$ be the $\boldsymbol{k}$-linear functional defined by the same formula in Lemma~\ref{2-9}. 
Then for a braided Hopf algebra $(A_{NL}^{\nu \lambda }, \sigma )$ the following statements  hold. 
\par 
$(1)$ Let $\alpha , \beta$ be elements in $\boldsymbol{k}^{\times}$ satisfying  $(\alpha \beta )^N=\nu ,\ (\alpha \beta ^{-1})^L=\lambda $. 
Then, 
$\theta _{\omega }$ is extended to a coribbon element of the braided bialgebra $(A_{NL}^{\nu \lambda },\ \sigma _{\alpha \beta })$ if and only if 
$\omega^{2N}=\alpha ^{2N}$, 
and any coribbon element of the braided  bialgebra $(A_{NL}^{\nu \lambda },\ \sigma _{\alpha \beta })$ is given by the form $\theta _{\omega }$. 
In addition, 
 $\theta _{\omega }$ is a  coribbon element of the  braided Hopf algebra  $(A_{NL}^{\nu \lambda },\ \sigma _{\alpha \beta })$ 
if and only if $\omega =\pm \beta^{-1}$. 
Therefore, there are exactly two coribbon elements of the 
braided Hopf algebra $(A_{NL}^{\nu \lambda },\ \sigma _{\alpha \beta })$. 
\par 
$(2)$ Let $\gamma , \delta $ be elements in $\boldsymbol{k}^{\times}$ satisfying  $\gamma ^2=\delta ^2,\ \gamma ^{2N}=1$. 
Then, 
$\theta _{\omega }$ is extended to a coribbon element of the braided bialgebra $(A_{N2}^{\nu \lambda },\tau _{\gamma \delta }^{\lambda})$ if and only if   
$\omega ^{2N}=1$, 
and any coribbon element of the braided  bialgebra $(A_{N2}^{\nu \lambda },\tau _{\gamma \delta }^{\lambda })$ is given by the form $\theta _{\omega}$. 
In addition, 
 $\theta _{\omega }$ is a  coribbon element of the  braided Hopf algebra  $(A_{N2}^{\nu \lambda },\ \tau _{\gamma \delta }^{\lambda})$ 
if and only if $\omega =\pm \gamma^{-1}$. 
Therefore, there are exactly two coribbon elements of the 
braided Hopf algebra $(A_{N2}^{\nu \lambda },\ \tau _{\gamma \delta }^{\lambda})$. 
\end{thm}
\begin{proof} 
(1) Let $\theta $ be a coribbon element of the braided bialgebra $(B/\langle J_{NL}^{\nu \lambda }\rangle $, $\sigma _{\alpha \beta })$. 
By $\Delta (x_{1j})=x_{11}\otimes x_{1j}+x_{12}\otimes x_{2j}$ for $j=1,2$ 
and (CR1) we have 
\par \smallskip \centerline{
$\begin{cases}
\theta (x_{11})x_{11}+\theta (x_{12})x_{21}=\theta (x_{11})x_{11}+\theta (x_{21})x_{12}, \\ 
\theta (x_{11})x_{12}+\theta (x_{12})x_{22}=\theta (x_{12})x_{11}+\theta (x_{22})x_{12}. \end{cases}$}
\par \smallskip \noindent 
Since 
$x_{11},\ x_{12},\ x_{21}=\nu x_{12}^{2N}x_{21},\ x_{22}=x_{11}^{2N}x_{22}$
are linearly independent,  
it follows that 
$\theta (x_{12})=\theta (x_{21})=0,\ \theta (x_{11})=\theta (x_{22})$. 
One can set 
$\theta (x_{11})=\theta (x_{22})=\omega $ for some $\omega \not= 0$ since 
$\theta $ is convolution-invertible. 
So, $\theta $ is obtained by  
$\theta =\tilde{\theta }_{\omega }\circ \pi $, 
where 
$\pi :B\longrightarrow B/\langle J_{NL}^{\nu \lambda }\rangle $ is the natural projection, and $\tilde{\theta }_{\omega }$ is the coribbon element of the braided bialgebra $(B, \sigma _{\alpha \beta })$ determined by $\tilde{\theta }_{\omega }(x_{ij})=\delta _{ij}\omega $ for all $i,j=1,2$. 
Thus, by Lemma~\ref{2-9}(1) it is required that $\omega =\xi \alpha \ $ for some $\xi^{2N}=1$. 
It can be easily shown that the converse is true. 
By Lemma~\ref{2-9}(2) a necessary and sufficient condition for that 
$\theta _{\omega }$ is a coribbon element of  $(A_{NL}^{\nu \lambda}, \sigma _{\alpha \beta})$ is $\omega =\pm \beta ^{-1}$. 
\par 
(2) Let $\theta $ be a coribbon element of the braided bialgebra $(B/\langle J_{N2}^{\nu \lambda }\rangle ,\ \tau _{\gamma \delta  }^{\lambda })$. 
As the same manner with the proof of Part (1) 
we see that 
$\theta (x_{12})=\theta (x_{21})=0,\ \theta (x_{11})=\theta (x_{22})$, 
and $\theta (x_{11})=\theta (x_{22})=\omega $ is not $0$. 
Hence $\theta $  is given by 
$\theta =\tilde{\theta }_{\omega }\circ \pi '$, 
where 
$\pi ':B^{(\lambda )}\longrightarrow B^{(\lambda )}/\langle J_N^{(\nu )}\rangle \cong B/\langle J_{N2}^{\nu \lambda }\rangle $ is the natural projection, $J_N^{(\nu )}=\boldsymbol{k}(x_{11}^2-x_{22}^2)+\boldsymbol{k}(x_{12}^2-x_{21}^2)+\boldsymbol{k}(x_{11}^{2N}+\nu x_{12}^{2N}-1)$, and 
$\tilde{\theta }_{\omega }$ is the coribbon element of the braided bialgebra $(B^{(\lambda )}, \tau _{\gamma \delta }^{\lambda })$. 
Thus, $\omega ^{2N}=1$ by Lemma~\ref{2-12}, and $\theta $ is needed to be the form in Part (2).  
The converse is also true. 
Furthermore, by Lemma~\ref{2-12}(2) a necessary and sufficient condition for that 
$\theta _{\omega }$ is a coribbon element of $(A_{N2}^{\nu \lambda}, \tau _{\gamma \delta }^{\lambda })$ is $\omega =\pm \gamma ^{-1}$. 
\end{proof} 

\par \smallskip 
To determine the coribbon elements of a 
braided Hopf algebra $(A_{NL}^{\nu \lambda}, \sigma )$ one can apply Corollary~\ref{4-4}. 
This fact gives us an alternative proof of Theorem~\ref{2-2} as follows. 
\par 
Suppose that $\boldsymbol{k}$ contains a $4NL$th root of  unity. 
If $L$ is odd, then 
$G((A_{NL}^{\nu \lambda})^{\ast})=
\{ \ p_{\omega },\ q_{\eta }\ \vert \ \omega , \eta \in \boldsymbol{k},\ \omega ^{2N}=1,\ \eta ^{2N}=\nu \ \} $. 
Here, 
$p_{\omega },\ q_{\eta } \in (A_{NL}^{\nu \lambda})^{\ast}$ are defined by 
$p_{\omega }(x_{ij})=\delta_{ij} \omega , \ q_{\eta }(x_{ij})=\delta _{i, j+1}\lambda ^j \eta $, 
 and  
the products between them are given by $p_{\omega }p_{\omega ^{\prime}}=p_{\omega \omega ^{\prime}},\ 
q_{\eta }q_{\eta  ^{\prime}}=p_{\lambda \eta \eta  ^{\prime}},\  
p_{\omega }q_{\eta }=q_{\eta }p_{\omega }=q_{\omega \eta }$. 
If $L$ is even, then  
\begin{align*}
G((A_{NL}^{\nu, -})^{\ast})
&=\{ \ p_{\omega , \epsilon }\ \vert \ \omega \in \boldsymbol{k},\ \omega ^{2N}=1,\ \epsilon =0,1\ \} ,\\ 
G((A_{NL}^{\nu , +})^{\ast})&=\{ \ p_{\omega , \epsilon },\ q_{\eta , \epsilon }\ \vert \ \omega , \eta \in \boldsymbol{k},\ \omega ^{2N}=1,\ \eta ^{2N}=\nu ,\ \epsilon =0,1\ \} .
\end{align*}
Here, 
$p_{\omega , \epsilon },\ q_{\eta , \epsilon } \in (A_{NL}^{\nu \lambda})^{\ast}$ are given by 
$p_{\omega , \epsilon }(x_{ij})=\delta_{ij}(-1)^{\epsilon (i-1)} \omega , \ 
q_{\eta , \epsilon }(x_{ij})=\delta _{i, j+1}(-1)^{\epsilon (i-1)} \eta $, 
and products between them are given by 
$p_{\omega , \epsilon }p_{\omega ^{\prime}, \epsilon ^{\prime}}= 
p_{\omega \omega ^{\prime}, \epsilon +\epsilon ^{\prime}}$,\  
$q_{\eta , \epsilon }q_{\eta  ^{\prime}, \epsilon ^{\prime}}=p_{(-1)^{\epsilon ^{\prime}}\eta \eta ^{\prime}, \epsilon +\epsilon ^{\prime}} , \ 
p_{\omega , \epsilon }q_{\eta , \epsilon ^{\prime}}=q_{\omega \eta , \epsilon +\epsilon ^{\prime}}, \ 
q_{\eta , \epsilon ^{\prime}}p_{\omega , \epsilon }=q_{(-1)^{\epsilon }\omega \eta , \epsilon +\epsilon ^{\prime}},$ 
where the indices of the right-hand sides are treated as modulo  $2$. 

\par 
\begin{prop}\label{4-8}
Suppose that $\boldsymbol{k}$ contains a $4NL$th root of unity. Then, 
$\text{Sph}((A_{NL}^{\nu \lambda})^{\ast})=\{\varepsilon ,\ p_{-1} \}$, 
where $\varepsilon $ is the counit of $A_{NL}^{\nu \lambda}$, and 
$p_{-1}$ is the algebra map defined by $p_{-1}(x_{ij})=-\delta _{ij}\ (i,j=1,2)$. 
\end{prop}
\begin{proof}
An element $p\in G((A_{NL}^{\nu \lambda})^{\ast})$ belongs to the center of $(A_{NL}^{\nu \lambda})^{\ast}$ if and only if $p(x_{11})=p(x_{22}),\ p(x_{12})=p(x_{21})=0$. 
Thus, whereas $p_{\omega }\in Z((A_{NL}^{\nu \lambda})^{\ast})$, $q_{\eta }\not \in Z((A_{NL}^{\nu \lambda})^{\ast})$ since $q_{\eta }(x_{12})=\eta \not= 0$. 
Furthermore, it follows from $p_{\omega }^2=p_{\omega ^2}$ that  
$p_{\omega }^2=\varepsilon $. This implies that $\omega ^2=1$, that is, $\omega =\pm 1$. 
Since $p_1=\varepsilon $, it follows that 
$\text{Sph}((A_{NL}^{\nu \lambda})^{\ast})=\{ \varepsilon ,\ p_{-1} \}$. 
\end{proof}

\par \medskip \noindent 
{\bf Alternative proof of Theorem~\ref{2-2}}.  
By Corollary~\ref{4-4} and Proposition~\ref{4-8} we have 
$\text{CRib}(A_{NL}^{\nu \lambda}, \sigma _{\alpha \beta })=\{  \varUpsilon_{\alpha \beta},  p_{-1}\varUpsilon_{\alpha \beta}\} $. 
Here, $\varUpsilon_{\alpha \beta}$ is the Drinfeld element of $(A_{NL}^{\nu \lambda}, \sigma _{\alpha \beta })$, and it is given by  
$ \varUpsilon_{\alpha \beta}(x_{ij})=\delta _{ij}\beta ^{-1}$ by Lemma~\ref{4-9}(1). 
Thus, (1) is proved. 
Similarly, it can be shown that $\text{CRib}(A_{NL}^{\nu \lambda}, \tau _{\gamma \delta }^{\lambda})=\{ \varUpsilon_{\gamma \delta}^{\lambda},  p_{-1}\varUpsilon_{\gamma \delta}^{\lambda} \} $, and hence (2) is also proved.  
\qed 

\par 
\section{Polynomial invariants for duals of Suzuki's braided Hopf algebras} 
In this section we assume that $N\geq 1,\ L\geq 2$, $\lambda , \nu =\pm 1$, and $\boldsymbol{k}$ is an algebraically closed field which contains a $4NL$th root of unity. 
We also assume that $\alpha , \beta \in \boldsymbol{k}^{\times}$ satisfy $(\alpha \beta )^N=\nu ,\ (\alpha \beta ^{-1})^L=\lambda$, and $\gamma , \delta \in \boldsymbol{k}^{\times}$ satisfy $\gamma ^2=\delta ^2,\ \gamma ^{2N}=1$. 
\par 
By Lemma~\ref{4-2} we have: 

\par  
\begin{lem}\label{4-0}
$(1)$ Let us consider the coribbon elements $\varUpsilon_{\alpha , \beta }$ and $\overline{\varUpsilon}_{\alpha , \beta }:=p_{-1}\varUpsilon_{\alpha , \beta }$ of the braided Hopf algebra $(A_{NL}^{\nu \lambda }, \sigma _{\alpha , \beta})$. 
\par 
$(i)$ for the simple right $A_{NL}^{\nu \lambda }$-comodule $\boldsymbol{k}g\ (g\in G(A_{NL}^{\nu \lambda }))$
\begin{align*}
&\xi _{\varUpsilon_{\alpha , \beta }}(\boldsymbol{k}g)
=\underline{\dim}_{\sigma _{\alpha \beta}} \boldsymbol{k}g \\
&\ =
\begin{cases}
\xi _{\overline{\varUpsilon}_{\alpha , \beta}}(\boldsymbol{k}g) =(\alpha \beta )^{-2s^2} & \text{if $g=x_{11}^{2s}\pm x_{12}^{2s}$}, \\[0.1cm]  
(-1)^L\xi _{\overline{\varUpsilon}_{\alpha , \beta}}(\boldsymbol{k}g) =(\alpha \beta )^{-2s^2-2sL-L^2}\alpha ^{L^2} & \text{if $g=x_{11}^{2s+1}\chi _{22}^{L-1}\pm \sqrt{\lambda}x_{12}^{2s+1}\chi_{21}^{L-1}$}.
\end{cases}  
\end{align*}
\indent 
$(ii)$ for the simple right $A_{NL}^{\nu \lambda }$-comodule  
$V_{st}=\boldsymbol{k}x_{11}^{2s}\chi_{22}^t+\boldsymbol{k}x_{12}^{2s}\chi_{21}^t$
$$\xi _{\varUpsilon_{\alpha , \beta }}(V_{st}) 
=\dfrac{\underline{\dim}_{\sigma _{\alpha \beta}} V_{st}}{2}=(\alpha \beta )^{-2s^2-2st-t^2}\alpha ^{t^2}
=(-1)^{t}\xi _{\overline{\varUpsilon}_{\alpha , \beta}}(V_{st}).$$
\indent 
$(2)$ Let us consider the coribbon elements $\varUpsilon_{\gamma \delta }^{\lambda}$ and $\overline{\varUpsilon}_{\gamma \delta }^{\lambda}:=p_{-1}\varUpsilon_{\gamma \delta }^{\lambda}$ of the braided Hopf algebra $(A_{N2}^{\nu \lambda }, \tau_{\gamma \delta }^{\lambda})$. 
\par 
$(i)$ for the simple right $A_{N2}^{\nu \lambda }$-comodule $\boldsymbol{k}g\ (g\in G(A_{N2}^{\nu \lambda }))$
$$\xi _{\varUpsilon_{\gamma \delta}^{\lambda }}(\boldsymbol{k}g)
=\xi _{\overline{\varUpsilon}_{\gamma \delta}^{\lambda }}(\boldsymbol{k}g)
=\underline{\dim}_{\tau _{\gamma \delta}^{\lambda }} \boldsymbol{k}g 
=
\begin{cases}
\gamma ^{-4s^2} & \text{if $g=x_{11}^{2s}\pm x_{12}^{2s}$}, \\[0.1cm]  
\gamma ^{-4(s+1)^2}\lambda 
 & \text{if $g=x_{11}^{2s+1}x_{22}\pm \sqrt{\lambda}x_{12}^{2s+1}x_{21}$}.
\end{cases} 
$$
\indent 
$(ii)$ for the simple right $A_{N2}^{\nu \lambda }$-comodule 
$V_{s1}=\boldsymbol{k}x_{11}^{2s}x_{22}+\boldsymbol{k}x_{12}^{2s}x_{21}$
$$\xi _{\varUpsilon_{\gamma \delta}^{\lambda }}(V_{s1}) 
=-\xi _{\overline{\varUpsilon}_{\gamma \delta}^{\lambda }}(V_{s1})
=\dfrac{\underline{\dim}_{\tau _{\gamma \delta}^{\lambda }} V_{s1}}{2}
=\gamma ^{-(2s+1)^2}.$$
\end{lem}
\begin{proof}
In the case of $\xi _{\varUpsilon_{\alpha , \beta }}$ and $\xi _{\varUpsilon_{\gamma , \delta }^{\lambda}}$, 
by Lemma~\ref{3-10}, 
$\xi _{\varUpsilon_{\alpha , \beta }}(M)$ \newline 
$= \underline{\dim}_{\sigma _{\alpha \beta}} M/\dim M$,\ 
$\xi _{\varUpsilon_{\gamma , \delta }^{\lambda}}(M)
= \underline{\dim}_{\tau _{\gamma , \delta }^{\lambda}} M/ \dim M$ for any absolutely simple right $A_{NL}^{\nu \lambda}$-comodule $M$. 
The values $\underline{\dim}_{\sigma _{\alpha \beta}} M$,\ $\underline{\dim}_{\tau _{\gamma , \delta }^{\lambda}} M$ have already computed in \cite[Lemma 5.9(1)]{W} although it needs to remove $\nu ^t$ from that formula. 
So, we obtain the formulas for $\xi _{\varUpsilon_{\alpha , \beta }}$ and $\xi _{\varUpsilon_{\gamma , \delta }^{\lambda}}$  in the proposition. 
Other equations can be  derived as follows.  
\par 
(1) (i) First, we note that $\text{ch}(\boldsymbol{k}g)=g$ for $g\in G(A_{NL}^{\nu \lambda })$. 
By Lemma~\ref{4-2}(2), 
if $g=x_{11}^{2s}\pm x_{12}^{2s} $, then 
$\xi _{\overline{\varUpsilon}_{\alpha , \beta}}(\boldsymbol{k}g)
=\overline{\varUpsilon}_{\alpha , \beta}(g) =p_{-1}(g)\varUpsilon _{\alpha , \beta}(g) =\xi _{\varUpsilon_{\alpha , \beta}}(\boldsymbol{k}g)$, 
and similarly if $g=x_{11}^{2s+1}\chi _{22}^{L-1}\pm \sqrt{\lambda}x_{12}^{2s+1}\chi_{21}^{L-1}$, then 
$\xi _{\overline{\varUpsilon}_{\alpha , \beta}}(\boldsymbol{k}g)
=(-1)^{L}\xi _{\varUpsilon_{\alpha , \beta}}(\boldsymbol{k}g)$. 
\par 
(ii) Since $\text{ch}(V_{st})=x_{11}^{2s+1}\chi_{22}^{t-1}+x_{11}^{2s}\chi_{22}^t$, we have 
\begin{align*}
\xi _{\overline{\varUpsilon}_{\alpha , \beta}}(V_{st})
&=\dfrac{\overline{\varUpsilon}_{\alpha , \beta}(\text{ch}(V_{st}))}{2}
=\dfrac{\bigl( p_{-1}\varUpsilon_{\alpha , \beta}\bigr) ( x_{11}^{2s+1}\chi_{22}^{t-1}+x_{11}^{2s}\chi_{22}^t)}{2} \\ 
&=\dfrac{p_{-1}(x_{11}^{2s+1}\chi_{22}^{t-1})\varUpsilon_{\alpha , \beta}(x_{11}^{2s+1}\chi_{22}^{t-1})
+p_{-1}(x_{12}^{2s+1}\chi_{21}^{t-1})\varUpsilon_{\alpha , \beta}(x_{21}^{2s+1}\chi_{12}^{t-1})}{2} \\ 
& \quad +\dfrac{p_{-1}(x_{11}^{2s}\chi_{22}^t)\varUpsilon_{\alpha , \beta}(x_{11}^{2s}\chi_{22}^t)
+p_{-1}(x_{12}^{2s}\chi_{21}^t)\varUpsilon_{\alpha , \beta}(x_{21}^{2s}\chi_{12}^t)}{2} \displaybreak[0]\\ 
&=(-1)^{2s+t}\dfrac{\varUpsilon_{\alpha , \beta}(\text{ch}(V_{st}))}{2}
=(-1)^{2s+t}\xi _{\varUpsilon_{\alpha , \beta}}(V_{st})
=(-1)^{t}\dfrac{\underline{\dim}_{\sigma _{\alpha ,\beta }}V_{st}}{2}. 
\end{align*}
\par 
By a similar computation we have the equations of (2). 
\end{proof} 

\par \smallskip 
By Lemma~\ref{4-0} we have: 

\par 
\begin{thm} 
$(1)$ For $i\in \{ 0,1,\ldots , N-1\}$ and $j\in \{ 0,1,\ldots , L-1\}$, set 
\par \medskip \centerline{
$\alpha _{ij\pm }:=\pm \omega ^{L(2i+\frac{1-\nu }{2})+N(2j+\frac{1-\lambda }{2})},\ 
\beta _{ij\pm }:=\pm \omega ^{L(2i+\frac{1-\nu }{2})-N(2j+\frac{1-\lambda }{2})}$.}
\par \smallskip \noindent 
Then $I_{\nu \lambda}:=\{ \ (\alpha ,\beta )\in \boldsymbol{k}\times \boldsymbol{k}\ \vert \ (\alpha \beta )^N=\nu ,\ (\alpha \beta ^{-1})^L=\lambda \ \} $ is 
represented as 
$I_{\nu \lambda}=\{ \ (\alpha _{ij+}, \beta _{ij+}),\ (\alpha _{ij-}, \beta _{ij-})\ |\ i=0,1,\ldots ,N-1,\ j=0,1,\ldots ,L-1\ \}$, 
and by setting $\epsilon _P=\epsilon ^{\prime}_P=1,\ \epsilon _Q=(-1)^L,\ \epsilon ^{\prime}_Q=(-1)^t$ and $X=P,Q$ we have
\begin{align*}
X_{(A_{NL}^{\nu \lambda })^{\ast}, \sigma _{\alpha _{ij\pm }, \beta _{ij\pm }}}^{(1)}(x)
&=\prod\limits_{s=1}^N\bigl( x- \omega ^{-4s^2L(2i+\frac{1-\nu }{2})}\bigr) \\ 
&\qquad \quad \cdot \bigl( x- \epsilon _X(\pm 1)^{L(j+1)}\omega ^{-(2s+L)^2L(2i+\frac{1-\nu }{2})+L^2N\frac{1-\lambda }{2}} \bigr)  ,\\ 
X_{(A_{NL}^{\nu \lambda })^{\ast}, \sigma _{\alpha _{ij\pm }, \beta _{ij\pm }}}^{(2)}(x)
&=\prod\limits_{s=0}^{N-1}\prod\limits_{t=1}^{L-1}( x-\epsilon ^{\prime}_X(\pm 1)^t\omega ^{-L(2s+t)^2(2i+\frac{1-\nu }{2})+t^2N(2j+\frac{1-\lambda }{2})}).
\end{align*}
\indent  
$(2)$ For $i\in \{ 0,1,\ldots , 2N-1\}$, define $\gamma _{i\pm}, \delta _{i\pm}\in \boldsymbol{k}^{\times}$ by $(\gamma _{i\pm}, \delta _{i\pm})=(\omega ^{4i},\ \pm \omega ^{4i})$. 
Then $J:=\{ \ (\gamma , \delta )\in \boldsymbol{k}\times \boldsymbol{k}\ \vert \ \gamma ^2=\delta ^2,\ \gamma ^{2N}=1 \ \} $ is represented as 
$J=\{\ (\gamma _{i+}, \delta _{i+}),\ (\gamma _{i-}, \delta _{i-})\ | \ i=0,1,\ldots ,2N-1\  \} $, 
and 
\begin{align*}
\hspace{0.3cm} P_{(A_{N2}^{\nu \lambda })^{\ast}, \tau _{\gamma _{i\pm}, \delta _{i\pm}}^{\lambda}}^{(1)}(x)
&=Q_{(A_{N2}^{\nu \lambda })^{\ast}, \tau _{\gamma _{i\pm}, \delta _{i\pm}}^{\lambda}}^{(1)}(x)
=\prod\limits_{s=1}^N(x-\omega ^{-16is^2})(x-\omega ^{-16i(s+1)^2}\lambda )  ,\displaybreak[0]\\ 
P_{(A_{N2}^{\nu \lambda })^{\ast}, \tau _{\gamma _{i\pm}, \delta _{i\pm}}^{\lambda}}^{(2)}(x)
&=\prod\limits_{s=0}^{N-1}( x-\omega ^{-4i(2s+1)^2}) ,\displaybreak[0]\\ 
Q_{(A_{N2}^{\nu \lambda })^{\ast}, \tau _{\gamma _{i\pm}, \delta _{i\pm}}^{\lambda}}^{(2)}(x)
&=\prod\limits_{s=0}^{N-1}( x+\omega ^{-4i(2s+1)^2}). 
\end{align*}
\end{thm} 

\par 
\begin{exam}\label{4-12}
Let $\omega \in \boldsymbol{k}$ be a primitive $8$th root of unity.  
For a braiding $\sigma $ of $H_8$ we set $P_{\sigma }^{(2)}(x):=P_{H_8, \sigma }^{(2)}(x)$. Then 
\begin{align*}
P_{\sigma _{\pm \omega , \pm \omega ^{-1}}}^{(2)}(x)=x\mp \omega ,\  & \quad 
P_{\sigma _{\pm \omega ^3, \pm \omega ^{-3}}}^{(2)}(x)=x\mp \omega ^3,\\      
P_{\tau _{1, \pm 1}^{-}}^{(2)}(x)=x-1,\  & \quad 
P_{\tau _{-1, \pm 1}^{-}}^{(2)}(x)=x+1. 
\end{align*}
It follows that all pairs of $(H_8, \sigma _{\omega , \omega ^{-1}})$,\ 
$(H_8, \sigma _{-\omega , -\omega ^{-1}}),\ 
(H_8, \sigma _{\omega ^3, \omega ^{-3}})$,\ \newline  
$(H_8, \sigma _{-\omega ^3, -\omega ^{-3}})$,\ 
$(H_8, \tau _{1, 1}^{-})$,\ 
$(H_8, \tau _{-1, 1}^{-})$ are not braided Morita equivalent. 
\end{exam}

\section{The braided Morita equivalence classes of $H_8$}
In this section we compute the automorphism group of the $8$-dimensional Kac-Paljutkin algebra $H_8$, and 
determine its braided Morita equivalence classes. 
\par 
Since the finite group $G_{12}$ is isomorphic to the dihedral group $D_8=\langle \ t,\ w\ \vert \ t^2=w^4=1,\ tw=w^{-1}t  \ \rangle $ of order $8$, 
it follows that $H_8$ is isomorphic to the group algebra  $\boldsymbol{k}D_8$ as an algebra. 
An algebra isomorphism 
$\varphi :\boldsymbol{k}D_8\longrightarrow H_8$ is given by 
$\varphi (t)=x_{12}+x_{22}, \ \varphi (w)=x_{11}x_{22}-x_{21}x_{12}$.
The induced Hopf algebra structure of  $\boldsymbol{k}D_8$ that    
$\varphi$ is a Hopf algebra map is as follows \cite{W}. 
\begin{align*}
\Delta (t)&=w^{-1}t\otimes e_1t+t\otimes e_0t,\ & 
\Delta (w)&=w\otimes e_0w+w^{-1}\otimes e_1w, \\ 
\varepsilon (t)& =1,\quad  & \varepsilon (w)&=1,\\ 
 S(t)& =(e_0-e_1w) t, \ &  
S(w)& =w, 
\end{align*}
where $e_0:=\frac{1+w^2}{2},\ e_1:=\frac{1-w^2}{2}$. 
They are central orthogonal idempotents, and satisfy
$
\Delta (e_0)=e_0\otimes e_0+e_1\otimes e_1,\  
\Delta (e_1)=e_0\otimes e_1+ e_1\otimes e_0$. 
Via the map $\varphi$ we identify $H_8=\boldsymbol{k}D_8$. 
Then 
$$x_{11}=\dfrac{w+w^{-1}}{2}t,\ \ 
x_{12}=\dfrac{1-w^2}{2}t,\ \ 
x_{21}=\dfrac{w^{-1}-w}{2}t,\ \ 
x_{22}=\dfrac{1+w^2}{2}t.$$
By \eqref{eq_grouplikes_of_SuzukiHopf} we see that the group-like elements of $H_8$ are given by 
\begin{equation*}
G(H_8)=\{ \ 1,\ w^2,\ w(e_0+\sqrt{-1}e_1),\ w(e_0-\sqrt{-1}e_1)\ \} \cong \mathbb{Z}/2\mathbb{Z}\oplus \mathbb{Z}/2\mathbb{Z}. 
\end{equation*}
We set 
$a:=w^2$ and $b:=w(e_0+\sqrt{-1}e_1)$. 
\par 
Let $f$ be a Hopf algebra automorphism on $H_8$. 
Then we see that 
\begin{align*}
f(e_i)
&=\frac{1}{2}\bigl(1+(-1)^if(a)\bigr) \ \ (i=0,1),\\ 
f(w^{\pm 1})&=\frac{1\mp \sqrt{-1}}{2}f(b)+\frac{1+\pm \sqrt{-1}}{2}f(ab), \\ 
f(e_0-e_1w)&=\frac{1}{2}\bigl(1+f(a)\bigr) +\frac{\sqrt{-1}}{2}\bigl( f(b)-f(ab)\bigr) . 
\end{align*}

Now, we write $x:=f(t)\in H_8$ as 
$x=\sum_{i,j=0,1} a_{ij}w^ie_j+\sum_{i,j=0,1} b_{ij} w^ie_jt$\ $(a_{ij},  b_{ij}\in \boldsymbol{k})$. 
Then, 
\begin{equation*}
\varepsilon (x)=1\ \ \Longleftrightarrow\ \ a_{00}+a_{10}+b_{00}+b_{10}=1,
\end{equation*}
\begin{equation*}
x^2=1,\ xw=w^{-1}x \ \ \Longleftrightarrow\ \ 
\begin{cases}
a_{01}=a_{11}=0, & a_{00}a_{10}+b_{00}b_{10}=0,\\ 
a_{00}^2+a_{10}^2+b_{00}^2+b_{10}^2=1,&  a_{00}b_{00}+a_{10}b_{10}=0,\\ 
b_{01}^2+b_{11}^2=1, & a_{10}b_{00}+a_{00}b_{10}=0.
\end{cases}
\end{equation*}

By solving the above equations,  
$x=f(t)$ is one of the following.

\begin{enumerate}
\item[(i)] $x=w^ie_0t^j+b_{01}e_1t+b_{11}we_1t\ \ (i, j=0,1)$
\item[(ii)] $x=\dfrac{1}{2}(e_0+we_0)t^j+\dfrac{(-1)^i}{2}(e_0-we_0)t^{j+1}+b_{01}e_1t+b_{11}we_1t \ \ (i, j=0,1)$,
\end{enumerate}

\noindent 
where $b_{01}^2+b_{11}^2=1$ is satisfied for all cases. 
In these $x$, we search $f$ so that $S(x)=f(e_0-e_1w)x$. 
Then we see that $f$ is identical on $G$, or coincides with $f_1$ on $G$ defined by $f_1(a)=a, f_1(b)=ab$.  
Furthermore, it can be shown that 
$\Delta (x)=\bigl( f(w^{-1})\otimes f(e_1)+1\otimes f(e_0)\bigr) (x\otimes x)$ is satisfied if and only if 
$x=e_0t\pm e_1t=t, w^2t$ for $f|_G=\text{id}_G$, and $x=we_0t\pm we_1t =w^{\pm 1}t$ for $f|_G=f_1$. 
In this way we have:

\noindent 
\begin{lem}
If $f$ is a Hopf algebra automorphism on $H_8$, then 
$f$ is one of the Hopf algebra automorphisms $\text{id}_{H_8}, f_+, f_-, f_{+-}:=f_+\circ f_-$, where $f_{\pm }$ are defined by 
$f_{\pm}(w)=w^{-1},\  f_{\pm}(t)=w^{\pm 1}t$. 
Therefore, the group $\text{Aut}(H_8)$ of the Hopf algebra automorphisms is 
\par \smallskip \centerline{\hspace{1.5cm} $\text{Aut}(H_8)=\{ \text{id}_{H_8},\ f_{+},\ f_{-},\ f_{+-} \} \cong \mathbb{Z}/2\mathbb{Z}\oplus \mathbb{Z}/2\mathbb{Z}.$\qed } 
\end{lem} 

Since $f_{+}(x_{11})=x_{22},\   
f_{+}(x_{12})=-x_{21}$,\   
$f_{+}(x_{21})=-x_{12},\   
f_{+}(x_{22})=x_{11}$,  
it follows that $\tau_{1,\pm 1}^{-}\circ (f_{+}\otimes f_{+})=\tau_{1,\mp 1}^{-}$,  and this implies the following result. 

\par 
\begin{cor}\label{3-2} 
As braided Hopf algebras $(H_8, \tau_{1,1}^-)\cong (H_8, \tau_{1,-1}^-),\ 
(H_8, \tau_{-1,1}^-)$ $\cong (H_8, \tau_{-1,-1}^-)$. 
In particular, there are isomorphisms  
${}_{(H_8, \tau_{1,1}^-)}\mathbb{M}\cong {}_{(H_8, \tau_{1,-1}^-)}\mathbb{M}$ and 
${}_{(H_8, \tau_{-1,1}^-)}\mathbb{M}\cong {}_{(H_8, \tau_{-1,-1}^-)}\mathbb{M}$ 
as $\boldsymbol{k}$-linear braided monoidal categories. \qed 
\end{cor}

\par 
By Corollary~\ref{3-2} and Example~\ref{4-12} we have:

\par 
\begin{thm}
Let $\boldsymbol{k}$ be an algebraically closed field whose characteristic is not $2$.  
For two braidings $\sigma , \sigma ^{\prime}$ of the $8$-dimensional Kac-Paljutkin algebra $H_8$ over 
$\boldsymbol{k}$, 
the braided Hopf algebras $(H_8, \sigma )$ and $(H_8, \sigma ^{\prime})$ are braided Morita equivalent if and only if one of the following is satisfied: 
\par \smallskip \centerline{$(1)$ $\sigma =\sigma ^{\prime}$,\quad 
$(2)$ $\{ \sigma , \sigma ^{\prime}\} =\{ \tau _{1, 1}^{-},\ \tau _{1, -1}^{-}\} $, \quad 
$(3)$ $\{ \sigma , \sigma ^{\prime}\} =\{ \tau _{-1, 1}^{-},\ \tau _{-1, -1}^{-}\} $.}
\par \smallskip \noindent 
Therefore, there are exactly $6$ braided Morita equivalence classes for $H_8$. 
\qed 
\end{thm}

\par \medskip \noindent 
{\bf Appendix: List of corrigenda in \cite{W2} with correct statements.}

\begin{enumerate}
\item[$\bullet$] p.333 in the abstract and p.334, l.6--7; the following sentence should be deleted: 
\par 
\colorbox[gray]{0.7}{As a consequence, we see that such a Hopf algebra has a coribbon }\newline 
\colorbox[gray]{0.7}{structure if and only if it is of Kac-Paljutkin type (see Theorem 5).}
\item[$\bullet$] p.339, the statements of Theorem 5 should be changed as follows. 
\par 
{\sc Theorem 5}. 
(1) The set of coribbon elements of the braided Hopf algebra $(A_{NL}^{\nu \lambda }, \sigma _{\alpha \beta })$ is $\{ \ \theta _{\beta^{-1}},\ \theta _{-\beta^{-1}} \} $. 
\par 
(2) The set of coribbon elements of the braided Hopf algebra $(A_{NL}^{\nu \lambda }, \tau _{\alpha \beta }^{\lambda })$ is $\{ \ \theta _{\beta^{-1}},\ \theta _{-\beta^{-1}} \} $. 
\par 
Here, $\theta _{\pm \beta ^{-1}}$ are the elements of $(A_{NL}^{\nu \lambda })^{\ast}$, that are determined by the condition (iii) in Definition 3 and the equations  
$\theta _{\pm \beta ^{-1}}(x_{ij})=\pm \delta _{ij}\beta ^{-1} \ (i,j=1,2)$.
\item[$\bullet$] p.340, the conclusion part of Lemma 8 (1) should be changed as follows: 
Then, $\theta _{\omega }$ induces a coribbon element of the bialgebra $(A_{NL}^{\nu \lambda } , \sigma _{\alpha \beta })$ if and only if $\omega ^{2N}=\alpha ^{2N}$. 
\item[$\bullet$] p.340, the symbols $\chi _1^L, \chi_2^L, \eta _1^L, \eta _2^L$ should be replaced by $\chi _{11}^L, \chi _{22}^L, \chi _{12}^L, \chi _{21}^L$, respectively, and \eqref{eq2-1} should be added. 
\item[$\bullet$] p.341, the parts from the fourth line to the 16th line should be modified as follows: 
\par 
If $m$ is even, then 
\begin{align*}
\sigma _{\alpha \beta }^{-1}(x_{12}^{m-1}, x_{12})
&=\sigma _{\alpha \beta }^{-1}(x_{21},x_{21}^{m-1})
=\alpha ^{-\frac{m}{2}}\beta ^{-(\frac{m}{2}-1)}, \\ 
\sigma _{\alpha \beta }^{-1}(x_{12}^{m-1}, x_{11})
&=\sigma _{\alpha \beta }^{-1}(x_{11}, x_{21}^{m-1})
=0. 
\end{align*}
If $m\geq 3$ is  odd, then 
\begin{align*}
\sigma _{\alpha \beta }^{-1}(x_{21}^{m-1}, x_{21})
&=\sigma _{\alpha \beta }^{-1}(x_{12},x_{12}^{m-1})=0,\\ 
\sigma _{\alpha \beta }^{-1}(x_{21}^{m-1}, x_{22})
&=\sigma _{\alpha \beta }^{-1}(x_{22}, x_{12}^{m-1})
=\alpha ^{-\frac{m-1}{2}}\beta ^{-\frac{m-1}{2}}.
\end{align*}
Hence, if $m$ is even, then 
\begin{align*}
\theta _{\omega }(x_{11}^{m})
&=\sigma _{\alpha \beta }^{-1}(x_{12}^{m-1},x_{12})\theta _{\omega }(x_{22}^{m-1})\theta _{\omega }(x_{22})\sigma _{\alpha \beta }^{-1}(x_{21},x_{21}^{m-1}) \\ 
&=\omega \tilde{\theta }_{\omega }(x_{22}^{m-1})(\alpha ^{-1})^{m}(\beta ^{-1})^{m-2}, 
\end{align*}
and if $m$ is odd, then 
\begin{align*}
\theta _{\omega }(x_{22}^{m})
&=\sigma _{\alpha \beta }^{-1}(x_{21}^{m-1},x_{21})\theta _{\omega }(x_{11}^{m-1})\theta _{\omega }(x_{11})\sigma _{\alpha \beta }^{-1}(x_{12},x_{12}^{m-1}) \\ 
&=\omega \tilde{\theta }_{\omega }(x_{11}^{m-1})(\alpha ^{-1})^{m-1}(\beta ^{-1})^{m-1} . 
\end{align*}
Thus, we have 
\begin{align*}
\theta _{\omega }(x_{11}^{2N}+\nu x_{12}^{2N})
&=\omega \alpha ^{-2N}\beta ^{-(2N-2)}\theta _{\omega }(x_{22}^{2N-1}) \\ 
&=\omega ^2(\alpha ^{-1})^{2N+(2N-2)}(\beta ^{-1})^{(2N-2)+(2N-2)}\tilde{\theta }_{\omega }(x_{11}^{2N-2}) \\
&=\cdots \cdots \\ 
&=\omega ^{2N}\alpha ^{-2N}. 
\end{align*}
It follows that 
\par \smallskip \centerline{$\mbox{(i)}\ \ \Longleftrightarrow  \ \ \omega ^{2N}=\alpha ^{2N}.$}
\item[$\bullet$] p.342, 
the equations in the 14th and 16th lines should be modified as follows, respectively: 
\par 
\par \centerline{$\tilde{\theta }_{\omega }(x_{ij}^2)
=(\tilde{\tau }_{\alpha \beta }^{\lambda })^{-1}(x_{ii},x_{ii})\tilde{\theta }_{\omega }(x_{ij})\tilde{\theta }_{\omega  }(x_{ij})(\tilde{\tau }_{\alpha \beta }^{\lambda })^{-1}(x_{jj},x_{jj})
=\tilde{\theta }_{\omega }(x_{ij})^2\text{\colorbox[gray]{0.7}{$\beta ^{-2}$}},$}
\par \centerline{$(\tau _{\alpha \beta }^{\lambda })^{-1}(x_{jj}^{m-1}, x_{jj})=(\tau _{\alpha \beta }^{\lambda })^{-1}(x_{jj},x_{jj}^{m-1})=\text{\colorbox[gray]{0.7}{$\alpha ^{-(m-1)}$}}$}
\item[$\bullet$] p.342, the equation \lq\lq $=\omega \gamma ^{2(m-1)}\theta _{\omega }(x_{11}^{m-1})\theta _{\omega }(x_{22}^{m})$" in the 19th line should be modified as 
\lq\lq $=\omega \alpha ^{-2(m-1)}\theta _{\omega }(x_{11}^{m-1}),$" and the equation in the 21st line should be modified as 
\lq\lq $= \omega \text{\colorbox[gray]{0.7}{$\alpha ^{-2(m-1)}$}}\theta _{\omega }(x_{22}^{m-1}), $"
\item[$\bullet$] p.342, the equations in the 25th and 27th lines should be modified as follows, respectively: 
\par 
\par \centerline{$\theta _{\omega }(x_{11}^{m})=\theta _{\omega }(x_{22}^{m})=\omega ^m\text{\colorbox[gray]{0.7}{$\alpha ^{-2((m-1)+(m-2)+\cdots +1)}$}}=\omega ^m\text{\colorbox[gray]{0.7}{$\alpha ^{-m(m-1)}$}}.$}
\par  \centerline{$\mbox{(i)}\ \ \Longleftrightarrow  \ \ \omega ^{2N}\text{\colorbox[gray]{0.7}{$\alpha ^{-2N(2N-1)}$}}=1\ \ \Longleftrightarrow  \ \ \omega ^{2N}=1.$}
\item[$\bullet$] p.342, in the fourth line from the bottom 
the sentence 
\lq\lq $(A_{NL}^{\nu \lambda }\rangle , \sigma _{\alpha \beta })$," 
should be modified as 
\lq\lq $(A_{NL}^{\nu \lambda }, \sigma _{\alpha \beta })$ as a braided bialgebra," 
\item[$\bullet$] p.343, l.9; the statement 
\lq\lq By Lemma 8, it follows that $N=1$ and $\omega =\pm \alpha $." 
should be corrected as follows: 
By Lemma 8, it follows that  $\omega ^{2N}=\alpha ^{2N}$. 
Since $\theta _{\omega }(S(x_{ij}))=\delta _{ij}\omega ^{-1}\beta ^{-2}$,  
the condition $\theta _{\omega }\circ S=\theta _{\omega }$ implies $\omega =\pm \beta ^{-1}$. 
\item[$\bullet$] p.343, the part from the 17th line to the 18th line should be modified below:  
Therefore, by Lemma 8, it follows that $\omega ^{2N}=1$. 
Since $\theta _{\omega }(S(x_{ij}))=\delta _{ij}\omega ^{-1}\beta ^{-2}$,  
the condition $\theta _{\omega }\circ S=\theta _{\omega }$ implies $\omega =\pm \beta ^{-1}$. 
\end{enumerate}


\begin{thebibliography}{99}

\bibitem{CDMM}
{C.~C\u{a}linescu, S.~D\u{a}sc\u{a}lescu, A.~Masuoka and C.~Menini}, 
{\itshape Quantum lines over non-cocommutative cosemisimple Hopf algebras}, 
J. Algebra {\bfseries 273} (2004), 753--779.

\bibitem{Doi} 
{Y.~Doi}, 
{\itshape Braided bialgebras and quadratic bialgebras}, 
Comm. Algebra {\bf 21} (1993), 1731--1749. 

\bibitem{Dri}
{V.G.~Drinfel'd}, 
\newblock 
{\itshape Quantum groups}.
\newblock In {Proceedings of the International Congress of Mathematics, Berkeley, CA., 1987}, 798--820. 

\bibitem{EG} 
{P.~Etingof and S.~Gelaki}, 
{\itshape On finite-dimensional semisimple and cosemisimple Hopf algebras in positive characteristic}, 
I.M.R.N. {\bfseries no.16} (1998), 851--864.

\bibitem{G} 
{S.~Gelaki}, 
{\itshape On the classification of finite-dimensional triangular Hopf algebras}, in: 
\lq New directions in Hopf algebras' edited by S. Montgomery and H.-J. Schneider, 
MSRI Publications {\bfseries 43}, 2002, 69--116. 

\bibitem{GMN} 
{C.~Goff, G.~Mason and S.-H.~Ng}, 
{\itshape On the gauge equivalence of twisted quantum doubles of elementary abelian and extra-special $2$-groups}, 
J. Algebra {\bfseries 312} (2007), 849--875.

\bibitem{Hayashi} 
{T.~Hayashi}, 
{\itshape Quantum groups and quantum determinants}, 
J. Algebra {\bfseries 152} (1992), 146--165. 

\bibitem{Hayashi2} 
{T.~Hayashi}, 
{\itshape Coribbon Hopf (face) algebras generated by lattice models},
J. Algebra {\bfseries 233} (2000), 614--641. 

\bibitem{KP} 
{G.I.~Kac and V.G.~Paljutkin},
{\itshape Finite ring groups}, 
in: Transactions of the Moscow Mathematical Society for the year 1966, 
AMS, 1967, 251--294 (original Russian paper: Trudy Moscov. Math. Ob. 15 (1966), 224--261). 

\bibitem{Kassel} 
{C.~Kassel}, 
{\itshape Quantum Groups}, 
G.T.M. 155, Springer-Verlag, New York,1995. 

\bibitem{Kauff}
{L.H.~Kauffman}, 
{\itshape Gauss codes, quantum groups and ribbon Hopf algebras}, 
Reviews in Math. Phys. {\bfseries 5} (1993), 735--773.

\bibitem{KrillovJr}
{A.A.~Kirillov, Jr.}, 
{\itshape On an inner product in modular tensor categories}, 
J. Amer. Math. Soc. {\bfseries 9} (1996), 1135--1169. 

\bibitem{Majid} 
{S.~Majid}, 
{\itshape Representation-theoretic rank and double Hopf algebras}, 
Comm. Algebra {\bfseries 18} (1990), 3705--3712.

\bibitem{Masuoka0}
{A.~Masuoka}, 
{\itshape Semisimple Hopf algebras of dimension $6$,\ $8$}, 
Israel J. Math. {\bfseries 92} (1995), 361--373. 

\bibitem{Masuoka1}
{A.~Masuoka}, 
{\itshape Cocycle deformations and Galois objects for some cosemisimple Hopf algebras of finite dimension}, 
Contemp. Math. {\bfseries 267} (2000), 195--214. 

\bibitem{MontgomeryBook}
{S.~Montgomery}, 
{\itshape Hopf algebras and their action on rings}, C.B.M.S.{\bfseries 82}, 
American Mathematical Society, 1993. 

\bibitem{NN}
{D.~Naidu and D.~Nikshych}, 
{\itshape Lagrangian subcategories and braided tensor equivalences of twisted quantum doubles of finite groups}, 
Commun. Math. Phys. {\bfseries 279} (2008), 845--872.

\bibitem{R}
{D.E.~Radford}, 
{\itshape On the antipode of a quasitriangular Hopf algebra}, 
J.  Algebra {\bfseries 151} (1992), 1--11.

\bibitem{RT}
{N.Yu.~Reshetikhin and V.G.~Turaev}, 
{\itshape Ribbon graphs and their invariants derived from quantum groups}, 
Commun. Math. Phys. {\bfseries 127} (1990), 1--26.

\bibitem{Sommer}
{Y.~Sommerh\"{a}user}, 
{\itshape Remarks on \lq M. Wakui: The coribbon structures of some finite dimensional braided Hopf algebras generated by $2\times 2$-matrix coalgebras'}, 
a private note, 2009. 

\bibitem{Suzuki2}
{S.~Suzuki}, 
{\itshape A family of braided cosemisimple Hopf algebras of finite dimension}, 
Tsukuba J. Math. {\bfseries 22} (1998), 1--29.

\bibitem{VanOZ}
{F.~Van Oystaeyen and Y.~Zhang}, 
{\itshape The Brauer group of a braided monoidal category},
J. Algebra {\bfseries 202} (1998), 96--128.

\bibitem{W2}
{M.~Wakui}, 
{\itshape The coribbon structures of some finite dimensional braided Hopf algebras generated by $2\times 2$-matrix coalgebras}, 
Banach Center Publ. {\bfseries 61},  
Noncommutative geometry and quantum groups,  2003, 333--344. 

\bibitem{W} 
{M.~Wakui}, 
{\itshape Polynomial invariants for a semisimple and cosemisimple Hopf algebra of finite dimension}, 
J. Pure Appl. Algebra {\bfseries 214} (2010), 701--728. 

\bibitem{Yetter} 
{D.N.~Yetter}, 
{\itshape Framed tangles and a theorem of Deligne on braided deformations of Tannakian categories}, 
Contemp. Math. {\bfseries 134} (1992), 325--349.


\end{thebibliography}
\end{document}